\documentclass[reqno,11pt]{amsart}

\usepackage[letterpaper,hmargin=25mm,vmargin=25mm]{geometry}
\usepackage{amsmath, amssymb,amsthm}
\usepackage[foot]{amsaddr}
\usepackage{graphicx}
\usepackage{xcolor}
\usepackage[colorlinks, citecolor=darkblue, linkcolor=darkblue, urlcolor=darkblue]{hyperref}

\definecolor{darkblue}{rgb}{0.0,0.0,0.4}

\newtheorem{theorem}{Theorem}[section]
\theoremstyle{definition}

\newtheorem{lemma}[theorem]{Lemma}

\newtheorem{example}[theorem]{Example}
\theoremstyle{remark}
\newtheorem{remark}[theorem]{Remark}

\title[Chebyshev-Sylvester method]{An expository review of the Chebyshev-Sylvester method in prime number theory}
\author{Tsogtgerel Gantumur}
\address{$^1$McGill University}
\address{$^2$National University of Mongolia}
\address{$^3$Institute of Mathematics and Digital Technology, Mongolian Academy of Sciences}
\date{January 23, 2026}


\begin{document}

\begin{abstract}
This paper gives a self-contained expository and computational account
of the elementary methods developed by P.~L.~Chebyshev and
J.~J.~Sylvester to obtain explicit bounds for the distribution of
prime numbers.
The method replaces the M\"obius function in a basic convolution
identity by finitely supported arithmetic functions, or schemes, and
relates the resulting auxiliary functions to the Chebyshev function
$\psi(x)$ through the summatory logarithm
$T(x)=\log(\lfloor x\rfloor!)$.

We examine a number of schemes introduced by Chebyshev and Sylvester,
with particular emphasis on Sylvester's iterative refinement procedure.
For a fixed choice of terms, this procedure leads to an explicit
two-dimensional affine recurrence for upper and lower bounds on
$\psi(x)$, whose fixed point can be analyzed by elementary linear
algebra.
We then implement the procedure computationally and introduce a simple
one-parameter selection rule, the $\rho$--rule, for choosing which
terms to retain in the iteration.
Numerical exploration of the historical schemes reproduces Sylvester's
bounds and, in several cases, yields modest improvements of the
resulting constants.
The accompanying Python implementation is publicly available.
\end{abstract}

\maketitle

\tableofcontents

\section{Introduction}

The distribution of prime numbers has been a central object of study in number theory for centuries. Before the proof of the Prime Number Theorem using complex analysis, mathematicians developed "elementary" methods to establish rigorous bounds on the prime-counting function, $\pi(x)$. In the 1850s, P. L. Chebyshev made the first significant breakthrough by proving that the ratio $\pi(x)/(x/\log x)$ is bounded between two positive constants \cite{cheb1848,cheb1850}.
His method involved relating the prime-counting function to the more analytically tractable logarithm of the factorial function.

Decades later, J. J. Sylvester expanded on this work by introducing a variety of new auxiliary functions and, most importantly, developing a powerful iterative procedure to systematically refine the bounds obtained from any given function \cite{sylv1881,sylv1892}.
This ``bootstrapping'' method represented the high water mark of elementary techniques prior to the proof of the Prime Number Theorem.

This paper provides a self-contained exposition of the
Chebyshev--Sylvester method.
We begin by introducing the necessary concepts from the theory of
arithmetic functions.
We then detail the main framework, analyze the schemes of Chebyshev
and Sylvester, and explain the iterative refinement procedure.
An earlier version of these notes appeared in Mongolian on the
author's blog in 2018 \cite{GantumurBlog2018}.
The present article develops that exposition further.
A distinctive feature of the present version is the use of modern
computational tools to implement and systematically explore
Sylvester's procedure.
All computations were carried out in Python using the Jupyter
notebooks in the repository \cite{GantumurNumbers}.

\section{Arithmetic functions and prime counting}

To make the paper self-contained, we first introduce the essential concepts from the theory of arithmetic functions, which form the language of our exposition.

An \emph{arithmetic function} is a function $f: \mathbb{N} \to \mathbb{C}$ defined on the set of positive integers. A key operation between two arithmetic functions $f$ and $g$ is the \emph{Dirichlet convolution}, denoted by $f*g$, which is defined as:
\begin{equation}
(f*g)(n) = \sum_{d|n} f(d)g(n/d) = \sum_{ab=n} f(a)g(b).
\end{equation}
This operation is commutative and associative. 
The identity element for convolution is the function $\delta$, defined as $\delta(1)=1$ and $\delta(n)=0$ for $n>1$.

The fundamental theorem of arithmetic can be written in additive form as
\begin{equation}\label{e:lnn}
\ln n = \sum_{d|n} \Lambda(d) ,
\end{equation}
where $\Lambda(n)$ is the \emph{von Mangoldt function}
\begin{equation}
\Lambda(n) = \begin{cases} \ln p & \text{if } n=p^k \text{ for some prime } p \text{ and integer } k \ge 1, \\ 0 & \text{otherwise.} \end{cases}
\end{equation}
Introducing the \emph{unit function}, $\mathit1(n) = 1$ for all $n \in \mathbb{N}$,
\eqref{e:lnn} becomes $\ln=\mathit1*\Lambda$.
Thus if we can invert $\mathit1$ under convolution, we can express $\Lambda$ in terms of $\ln$.
The \emph{M\"obius function}, defined as
\begin{equation}
\mu(n) = 
\begin{cases} 
1 & \text{if } n=1, \\ (-1)^k & \text{if } n \text{ is a product of } k \text{ distinct primes}, \\ 0 & \text{if } n \text{ has a squared prime factor}, 
\end{cases}
\end{equation}
does exactly what we need:
\begin{equation}
(\mu * \mathit1)(n) = \sum_{d|n} \mu(d) = \delta(n).
\end{equation}
By convolving $\ln=\mathit1*\Lambda$ with $\mu$, we get
\begin{equation}
\mu * \ln = \mu * (\mathit1*\Lambda) = (\mu*\mathit1)*\Lambda = \delta*\Lambda = \Lambda .
\end{equation}
Let us display the important identities for later use.

\begin{lemma}
We have $\ln=\mathit1*\Lambda$ and $\Lambda=\mu*\ln$.
\end{lemma}

The central object of study in prime number theory is the {\em prime counting function}
\begin{equation}
\pi(x) = \sum_{p \le x} 1 ,
\end{equation}
where the notation $p\le x$ means that the sum is over all primes not exceeding $x$.
The {\em prime number theorem} (PNT), which was a conjecture at the time of Chebyshev-Sylvester work, posits that
\begin{equation}
\pi(x) \cdot \frac{\ln x}x \to 1\qquad\text{as}\quad x\to\infty .
\end{equation}
In the Chebyshev-Sylvester framework (and in fact in all standard approaches to PNT), 
we access $\pi(x)$ through the following more analytically friendly function:
\begin{equation}
\psi(x) = \sum_{p^m \le x} \ln p = \sum_{n \le x} \Lambda(n) ,
\end{equation}
where the notation $p^m\le x$ means that the sum is over all primes $p$ and all positive integers $m$ satisfying $p^m\le x$.
This was the second of two functions introduced by Chebyshev.
An interesting side observation is
\begin{equation}
e^{\psi(x)}=\prod_{p^m\leq x}p=\mathrm{lcd}\{n:n\leq x\} .
\end{equation}

\begin{lemma}
For all $0<\alpha<1$ and $ x>1$, we have
\begin{equation}
\psi(x)\leq\pi(x)\ln x\leq\frac{\psi(x)}\alpha+x^\alpha\ln x ,
\end{equation}
and hence the prime number theorem is equivalent to the statement $\psi(x) \sim x$.
\end{lemma}

\begin{proof}
For any $p$, the highest power with $p^m\leq x$ is $m=\big\lfloor\frac{\ln x}{\ln p}\big\rfloor$, and so
\begin{equation}
\psi(x)=\sum_{p^m\leq x}\ln p=\sum_{p\leq x}\Big\lfloor\frac{\ln x}{\ln p}\Big\rfloor\ln p\leq\sum_{p\leq x}\frac{\ln x}{\ln p}\ln p=\ln x\sum_{p\leq x}1=\pi(x)\ln x.
\end{equation}
On the other hand, we have
\begin{equation}
\begin{split}
\psi(x)
&\geq\displaystyle\sum_{x^\alpha<p\leq x}\ln p\geq\sum_{x^\alpha<p\leq x}\ln x^\alpha=\alpha\ln x\sum_{x^\alpha<p\leq x}1\\
&=\displaystyle\alpha\ln x\cdot\big(\pi(x)-\pi(x^\alpha)\big)\geq\alpha\ln x\cdot\big(\pi(x)-x^\alpha\big).
\end{split}
\end{equation}
The proof is complete.
\end{proof}

Chebyshev's idea was to relate $\psi(x)=\sum_{n\le x}\Lambda(n)$ to the more approachable function
\begin{equation}
T(x) = \sum_{n \le x} \ln n = \ln(\lfloor x \rfloor!) = x \ln x - x + O(\ln x) ,
\end{equation}
through $\ln=\mathit1*\Lambda$.
To this end, we need the following result.

\begin{lemma}\label{l:summa}
We have
\begin{equation}\label{e:summa}
\sum_{n\leq x}(f*g)(n) = \sum_{k\le x}f(k)\sum_{d\leq x/k}g(d) .
\end{equation}
\end{lemma}

\begin{proof}
We proceed as
\begin{equation}
\sum_{n\leq x}(f*g)(n)=\sum_{n\leq x}\sum_{d|n}f(n/d)g(d)=\sum_{k\geq1}\sum_{kd\leq x}f(k)g(d) =\sum_{k\geq1}f(k)\sum_{d\leq x/k}g(d) ,
\end{equation}
and note that the outer sum is restricted by $k\le x$, since for $k>x$ the inner sum is empty.
\end{proof}

Applying this to  $\ln=\mathit1*\Lambda$, we get
\begin{equation}\label{e:summa-log-psi}
T(x)=\sum_{k\le x}\psi(x/k) ,
\end{equation}
and to  $\Lambda=\mu*\ln$, we get
\begin{equation}\label{e:psi-summa-log}
\psi(x)=\sum_{k\le x}\mu(k)T(x/k) .
\end{equation}
This is the starting point of Chebyshev's method.

\section{Chebyshev's basic framework}

In principle, \eqref{e:psi-summa-log} expresses $\psi(x)$ in terms of $T(x)$,
but it involves the M\"obius function, which is quite unwieldy.
So suppose that we have a surrogate $\nu$ of $\mu$,
and convolve it with the identity $\mathit1*\Lambda=\ln$, to get
\begin{equation}
\nu*\mathit1*\Lambda=\nu*\ln .
\end{equation}
Then summation of the right hand side via Lemma~\ref{l:summa} yields
\begin{equation}\label{e:sylvester-1}
V(x)=\sum_{n\leq x}(\nu*\mathit1*\Lambda)(n)=\sum_{n\leq x}(\nu*\ln)(n)=\sum_{k\leq x}\nu(k)T(x/k) .
\end{equation}
Note that the first expression is the definition of $V(x)$, which we expect to be an imitation of $\psi$, if $\nu$ imitates $\mu$ well.
The sense in which $\nu$ should ``imitate'' $\mu$ is part of what need to figure out as we develop the theory.

Now let us apply our summation formula \eqref{e:summa} to the splitting $(\nu*\mathit1)*\Lambda$, to get
\begin{equation}\label{e:sylvester-2}
V(x)=\sum_{n\leq x}(\nu*\mathit1*\Lambda)(n)=\sum_{k\leq x}E(x/k)\Lambda(k) ,
\end{equation}
where we have introduced 
\begin{equation}
E(x)=\sum_{n\leq x}(\nu*\mathit1)(n)=\sum_{k\leq x}\Big\lfloor\frac{x}{k}\Big\rfloor\nu(k) .
\end{equation}
Chebyshev's original choice for $V$ is
\begin{equation}
V(x)=T(x)-T(x/2)-T(x/3)-T(x/5)+T(x/30) ,
\end{equation}
which corresponds to $\nu$ that is given by
\begin{equation}
\nu_*=\delta_1-\delta_2-\delta_3-\delta_5+\delta_{30} .
\end{equation}
Here $\delta_m(n)=\delta_{m,n}$, that is, $\delta_m(n)=1$ for $n=m$, and $\delta_m(n)=0$ otherwise.

What conditions does $\nu$ need to satisfy in order to give us a good estimate on $\psi(x)$ through the relations \eqref{e:sylvester-1}-\eqref{e:sylvester-2}? 
The following conditions suggest themselves naturally.
\begin{itemize}
\item
Supposing that $V(x)\approx\psi(x)$ in some sense,
since $T(x)=x\ln x+O(x)$, 
the right hand side of \eqref{e:sylvester-1} can be $O(x)$ only if $\sum_n\nu(n)/n=0$.
Chebyshev's scheme satisfies this condition.
\item
In order for $V(x)\approx\psi(x)$ to hold, 
in view of \eqref{e:sylvester-2},
it is necessary that $E(x)\approx1$ at least for small values of $x$.
Chebyshev's scheme satisfies $E(x)=1$ for $1\leq x<6$.
\item
Finally, for practicality, we require that $\nu(n)\neq0$ for only finitely many $n$.
\end{itemize}

\begin{theorem}\label{t:cheb}
(a) If the cancellation condition 
\begin{equation}\label{e:sylvester-3}
\sum_n\frac{\nu(n)}n=0 ,
\end{equation}
is satisfied, then we have
\begin{equation}
V(x)=A(\nu)x+O(\ln x) ,
\end{equation}
where 
\begin{equation}
A(\nu) = -\sum_{n}\frac{\nu(n)\ln n}{n} .
\end{equation}
(b) If $E(x)\leq1$ for $x\geq1$,
then $V(x)\le\psi(x)$.
\\(c) If
$E(x)\geq0$ for all $x\geq1$ and
$E(x)\geq1$ for $1\leq x<N$, then $V(x)\ge\psi(x)-\psi(x/N)$.
In addition, if we assume \eqref{e:sylvester-3}, then
\begin{equation}
\psi(x)\le\frac{N}{N-1}A(\nu)x+O(\ln^2\!x) .
\end{equation}
\end{theorem}

\begin{proof}
(a) Using Stirling's approximation $T(x)=x\ln x-x+O(\ln x)$ in \eqref{e:sylvester-1},
and taking into account the cancellation condition, we derive
\begin{equation}
V(x)=\sum_{n\leq x}\nu(n)T(x/n)=-\Big(\sum_{n}\frac{\nu(n)\ln n}{n}\Big) x+O(\ln x)=A(\nu)x+O(\ln x) .
\end{equation}
(b) It is immediate from \eqref{e:sylvester-2} that
\begin{equation}
V(x)=\sum_{k\leq x}E(x/k)\Lambda(k)\leq\sum_{k\leq x}\Lambda(k)=\psi(x) .
\end{equation}
(c) Since $\Lambda(k)\geq0$, the global nonnegativity of $E$ gives
\begin{equation}
\begin{split}
V(x)
&=\sum_{k\leq x}E(x/k)\Lambda(k)
\geq\sum_{x/N<k\leq x}E(x/k)\Lambda(k)\\
&\geq\sum_{x/N<k\leq x}\Lambda(k)
 =\psi(x)-\psi(x/N).
\end{split}
\end{equation}
Then noting that
\begin{equation}
\begin{split}
\psi(x)-\psi(x/N)&\leq V(x)=A(\nu)x+O(\ln x),\\
\psi(x/N)-\psi(x/N^2)&\leq V(x/N)=A(\nu)x/N+O(\ln x),\\
\psi(x/N^2)-\psi(x/N^3)&\leq V(x/N^2)=A(\nu)x/N^2+O(\ln x),\\
&\ldots
\end{split}
\end{equation}
and summing over these $\lfloor\frac{\ln x}{\ln N}\rfloor$ lines gives us the upper bound 
\begin{equation}
\psi(x)\leq(1+N^{-1}+N^{-2}+\ldots)A(\nu)x+O(\ln^2\!x)=\frac{N}{N-1}A(\nu)x+O(\ln^2\!x) .
\end{equation}
\end{proof}

\begin{example}
For Chebyshev's scheme, we have
\begin{equation}
A_*=A(\nu_*)=\frac{\ln2}2+\frac{\ln3}3+\frac{\ln5}5-\frac{\ln30}{30}=0.92129\ldots .
\end{equation}
In particular, $E(x)$ has period 30, and $0\leq E(x)\leq1$ for all $x$.
Moreover, since $E(x)=1$ for $x<6$, we have $N=6$. 
Applying the preceding theorem then yields the familiar bounds
\begin{equation}
A_*x+O(\ln x)\leq\psi(x)\leq\frac65A_*x+O(\ln^2\!x) .
\end{equation}
Figure \ref{f:echeb} shows $E(x)$ corresponding to Chebyshev's scheme.
\end{example}

\begin{figure}[ht]
\centering
\includegraphics[width=.9\textwidth]{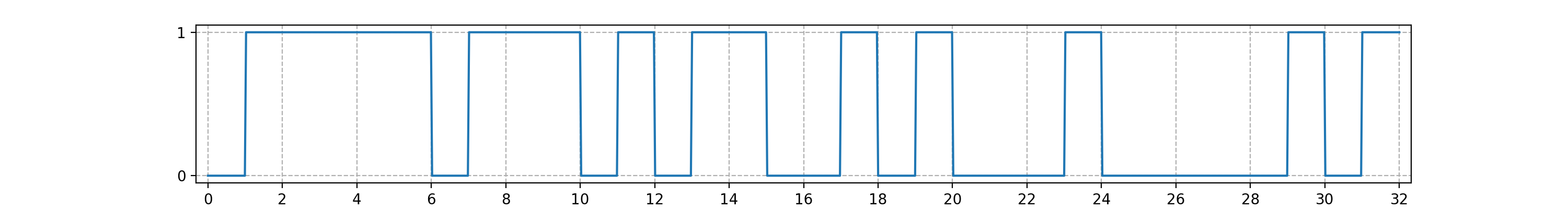}
\caption{$E(x)$ for Chebyshev's scheme.}
\label{f:echeb}
\end{figure}

\begin{remark}\label{r:graph-1}
Note that as \eqref{e:sylvester-2} means
\begin{equation}
V(x)=\sum_{n\leq x}E(x/n)\Lambda(n)=\sum_{n\leq x}\big(E(n)-E(n-1)\big)\psi(x/n) ,
\end{equation}
we can read off from the graph of $E(x)$ the coefficients of the expansion of $V(x)$ in terms of $\psi(x/n)$. 
For example, in Chebyshev's case we have
\begin{equation}
V(x)=\psi(x)-\psi(x/6)+\psi(x/7)-\psi(x/10)+\psi(x/11)-\psi(x/12)+\ldots,
\end{equation}
see Figure \ref{f:echeb1}.
\end{remark}

\begin{figure}[ht]
\centering
\includegraphics[width=.9\textwidth]{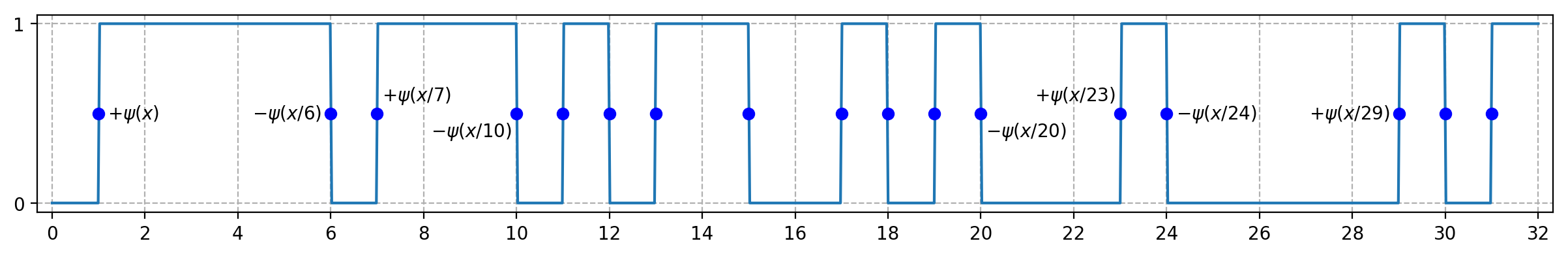}
\caption{Expanding $V(x)$ in terms of $\psi(x/n)$.}
\label{f:echeb1}
\end{figure}

We mention a closely related modern viewpoint due to Camargo and
Martin \cite{CamargoMartin2022,CamargoMartin2024}. In the notation
of their harmonic schemes, the floor-sum function $f_\mu$ is precisely
our function $E$, while the coefficients $b_n$ in the corresponding
expansion in $\psi(x/n)$ satisfy
\[
E(k)=\sum_{n\leq k}b_n,
\qquad\text{equivalently}\qquad
b_k=E(k)-E(k-1).
\]
They show, in particular, that the condition $E(k)\in\{0,1\}$ is
equivalent to the nonzero coefficients $b_k$ alternating between
$1$ and $-1$. Their principal emphasis is on connections with
constant components of the Mertens function, whereas our emphasis
below is on Sylvester's iterative refinement procedure and its
computational optimization.

\begin{remark}\label{r:graph-2}
On Figure \ref{f:echeb1}, in light of the nondecreasing property of $\psi(x)$, it is easy to see 
\begin{equation}\label{e:cheb-est-1}
\begin{split}
V(x)&=\psi(x)\underbrace{-\psi(x/6)+\psi(x/7)}_{\leq0}\underbrace{-\psi(x/10)+\psi(x/11)}_{\leq0}\underbrace{-\psi(x/12)+\ldots}_{\leq0}
\leq\psi(x) ,\\
V(x)&=\psi(x)-\psi(x/6)+\underbrace{\psi(x/7)-\psi(x/10)}_{\geq0}+\underbrace{\psi(x/11)-\psi(x/12)}_{\geq0}+\ldots \\
&\geq\psi(x)-\psi(x/6) .
\end{split}
\end{equation}
We illustrate this process in Figure \ref{f:echebe},
where the dashed red lines show the pairs that were canceled due to the nondecreasing property of $\psi(x)$.
\end{remark}

\begin{figure}[ht]
\centering
\includegraphics[width=.9\textwidth]{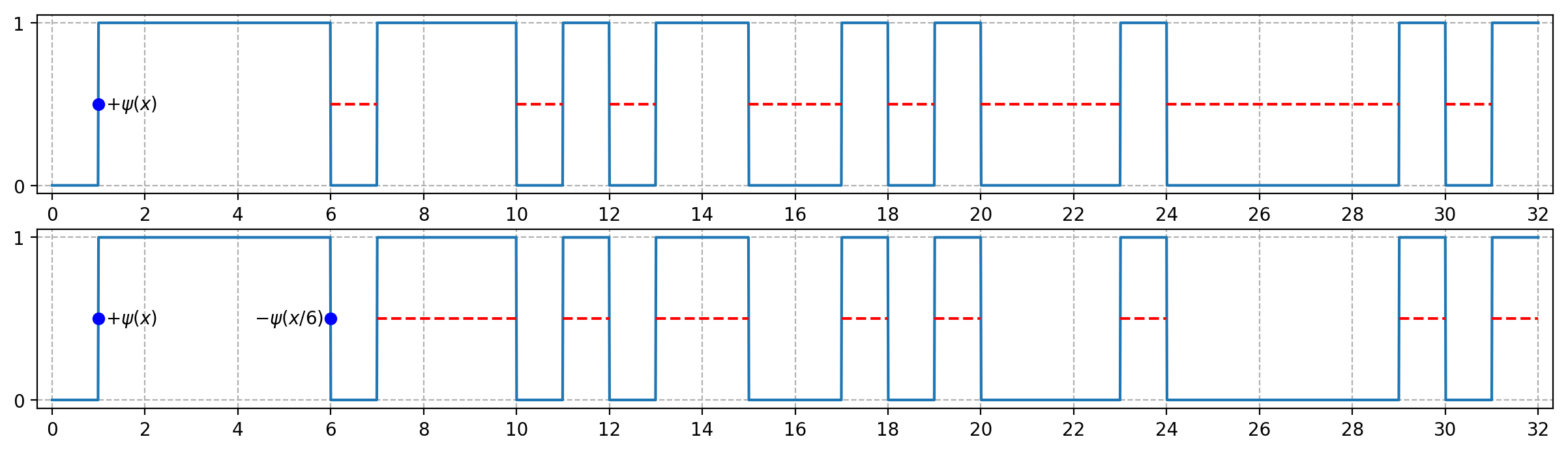}
\caption{Derivation of the estimates \eqref{e:cheb-est-1}.}
\label{f:echebe}
\end{figure}

\section{Sylvester's new schemes}

It had passed 30--40 years since Chebyshev's work,
when the following schemes were found by James Sylvester and his assistant James Hammond\footnote{Sylvester gives the credit of discovering $\nu_3$ to James Hammond.}:
\begin{equation}
\begin{split}
\nu_1&=\delta_1-2\delta_2,\\
\nu_2&=\delta_1-\delta_2-2\delta_3+\delta_6,\\
\nu_3&=\delta_1-\delta_2-\delta_3-\delta_4+\delta_{12} ,
\end{split}
\end{equation}
which are written in Sylvester's own notation as
\begin{equation}
\nu_1=[1;2,2],\qquad\nu_2=[1,6;2,3,3],\qquad\nu_3=[1,12;2,3,4] .
\end{equation}
In this notation, Chebyshev's scheme is
\begin{equation}
\nu_*=[1,30;2,3,5] .
\end{equation}
These new schemes satisfy $\eqref{e:sylvester-3}$, and $0\leq E(x)\leq1$ for all $x$. 
Moreover, we have the following.
\begin{itemize}
\item
For $\nu_1$, the period of $E(x)$ is 2, and $E(x)=1$ for $x<2$.
\item
For $\nu_2$, the period of $E(x)$ is 6, and $E(x)=1$ for $x<3$.
\item
For $\nu_3$, the period of $E(x)$ is 12, and $E(x)=1$ for $x<4$.
\end{itemize}
The corresponding functions $E(x)$ are shown in Figure \ref{f:e123}.

\begin{figure}[ht]
\centering
\includegraphics[width=.8\textwidth]{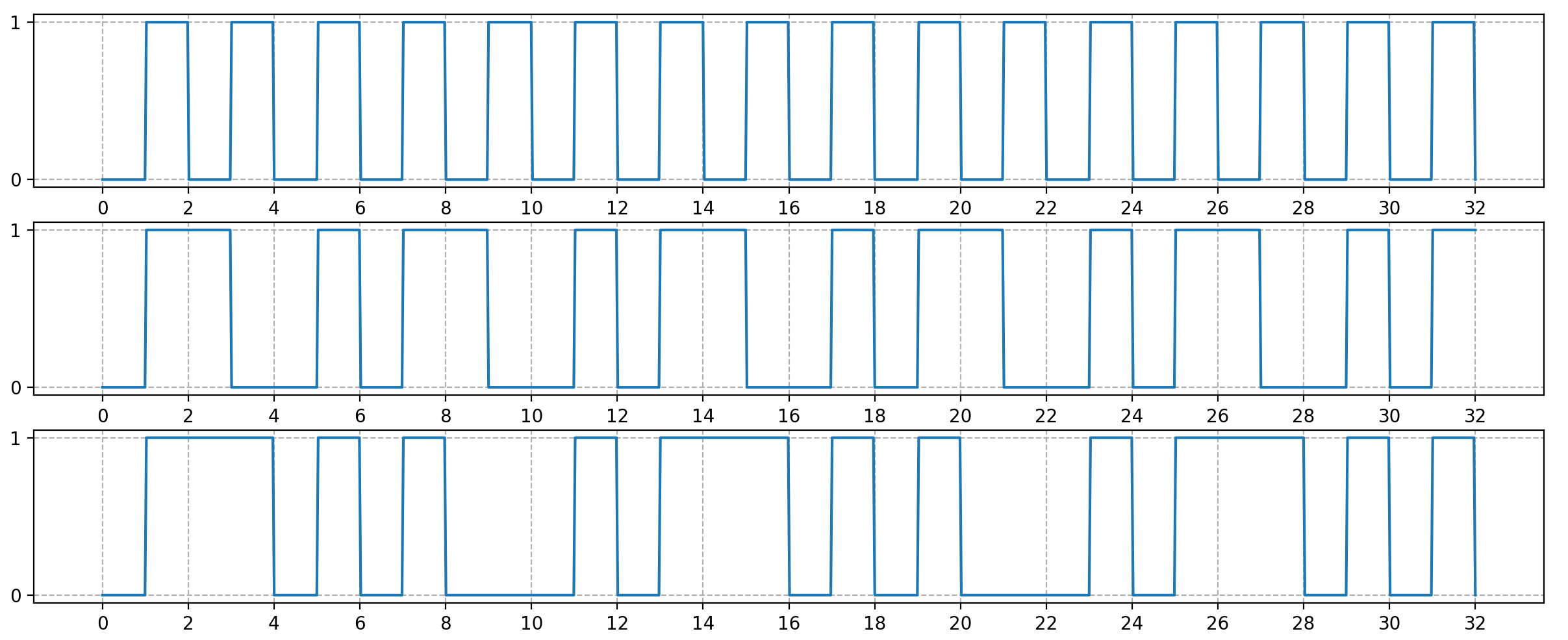}
\caption{The $E$-functions for $\nu_1$, $\nu_2$, and $\nu_3$.}
\label{f:e123}
\end{figure}

The constants $A(\nu)$ and $B(\nu)=\frac{N}{N-1}A(\nu)$
from Theorem~\ref{t:cheb}, for each new scheme, are as follows.
\begin{equation}
\begin{aligned}
A(\nu_1)&=0.6931\ldots,&\qquad B(\nu_1)&=2A(\nu_1)=1.3862\ldots,\\
A(\nu_2)&=0.7803\ldots,&\qquad B(\nu_2)&=\textstyle\frac32A(\nu_2)=1.1705\ldots,\\
A(\nu_3)&=0.8522\ldots,&\qquad B(\nu_3)&=\textstyle\frac43A(\nu_3)=1.1363\ldots
\end{aligned}
\end{equation}
Although all of these pairs are inferior to Chebyshev's
\begin{equation}
A(\nu_*)=0.9212\ldots,\qquad B(\nu_*)=\textstyle\frac65A(\nu_*)=1.1055\ldots ,
\end{equation}
each scheme is able to give $\psi(x)=O(x)$ and $x=O(\psi(x))$ without much effort,
and so might be of interest in some situations.
For instance, already the scheme $\nu_2$ is strong enough to imply Bertrand's postulate.

These four are particularly simple classical schemes satisfying
$0\leq E(x)\leq1$, but they are not the only schemes with this
property. Further examples of such harmonic schemes are given by
Camargo and Martin \cite{CamargoMartin2022}; see in particular
their Table~2.

In order to improve upon Chebyshev's constants, Sylvester considered more general schemes,
including 
\begin{equation}
\begin{split}
\nu_4&=\delta_1-\delta_2-\delta_3-\delta_6,\\
\nu_5&=\delta_1+\delta_{15}-\delta_2-\delta_3-\delta_5-\delta_{30},\\
\nu_6&=\delta_1+\delta_6+\delta_{70}-\delta_2-\delta_3-\delta_5-\delta_7-\delta_{210} ,
\end{split}
\end{equation}
or written in Sylvester's notation
\begin{equation}
\nu_4=[1,6;2,3],\qquad\nu_5=[1,15;2,3,5,30],\qquad\nu_6=[1,6,70;2,3,5,7,210] .
\end{equation}
They all satisfy \eqref{e:sylvester-3}, and $0\leq E(x)\leq2$ for all $x$.
Moreover, we have the following.
\begin{itemize}
\item
For $\nu_4$, the period of $E(x)$ is 6, $E(x)=1$ for $x<5$, and $N=6$.
\item
For $\nu_5$, the period of $E(x)$ is 30, $E(x)=1$ for $x<6$, and $N=6$.
\item
For $\nu_6$, the period of $E(x)$ is 210, $E(x)=1$ for $x<10$, and $N=10$.
\end{itemize}
The corresponding functions $E(x)$ are shown in Figure \ref{f:e456}.
For $\nu_6$, the period of $E(x)$ is 210, and hence Figure \ref{f:e456} does not contain a whole period.
This is remedied in Figure \ref{f:e5}.

\begin{figure}[ht]
\centering
\includegraphics[width=.7\textwidth]{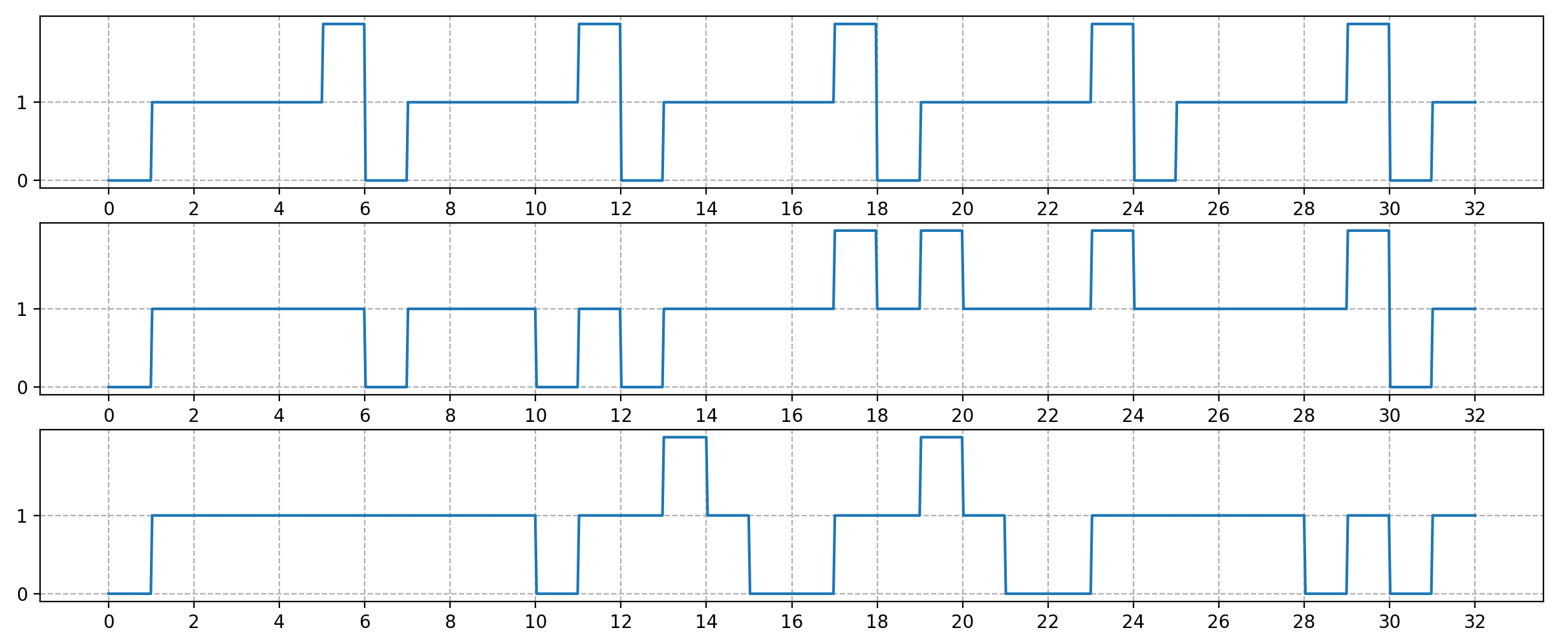}
\caption{The $E$-functions for $\nu_4$, $\nu_5$, and $\nu_6$.}
\label{f:e456}
\end{figure}

\begin{figure}[ht]
\centering
\includegraphics[width=.9\textwidth]{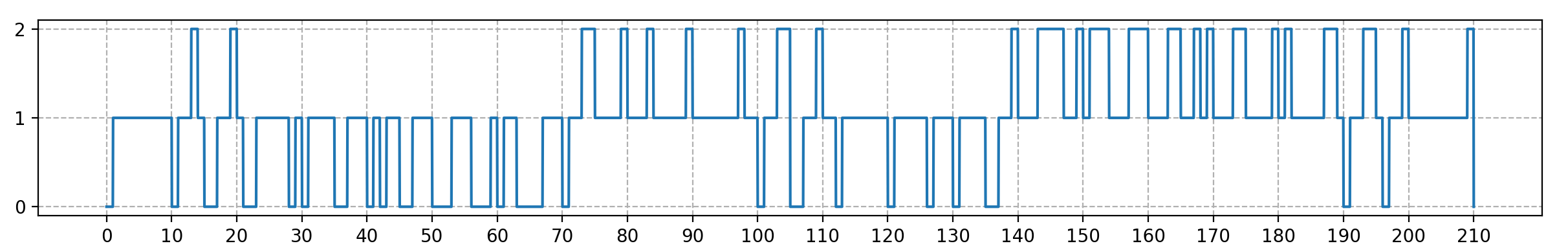}
\caption{A complete period of the $E$-function for $\nu_6$.}
\label{f:e5}
\end{figure}

The most striking feature of these schemes is that the condition $E\leq1$ is now violated, and hence Theorem~\ref{t:cheb}(b) cannot be applied.
However, (a) and (c) of that theorem are still valid, meaning that we have
\begin{equation}
V(x)=A(\nu)x+O(\ln x) ,
\end{equation}
and
\begin{equation}\label{e:sylvester-5}
\psi(x)\leq\frac{N}{N-1}A(\nu)x+O(\ln^2\!x) .
\end{equation}
The constants $A(\nu)$ corresponding to these schemes are
\begin{equation}
A(\nu_4)=1.0114\ldots,\quad A(\nu_5)=0.9675\ldots,\quad A(\nu_6)=0.9787\ldots.
\end{equation}
For a lower bound on $\psi$, we need to update Theorem~\ref{t:cheb}(b).

\begin{theorem}\label{t:sylv}
Assume the cancellation condition $\sum_n\nu(n)/n=0$, and assume
\[
0\leq E(x)\leq2\qquad\text{for all }x\geq1,
\]
with $E(x)\geq1$ for $1\leq x<N$.
In addition, assume $E(x)\leq1$ for $1\leq x<M$.
Then we have 
\begin{equation}
\psi(x)\ge\Big(1-\frac{N}{M(N-1)}\Big)A(\nu)x+O(\ln^2\!x) .
\end{equation}
\end{theorem}

\begin{proof}
The additional condition can be written as
\begin{equation}
E(x)\leq\chi(x)+\chi(x/M) ,
\end{equation}
where
\begin{equation}
\chi(x) = 
\begin{cases}
1, &\textrm{for}\quad x\geq1 ,\\
0, &\textrm{for}\quad x<1 .
\end{cases}
\end{equation}
Thus we have
\begin{equation}\label{e:sylv-upper}
V(x)=\sum_{k\leq x}E(x/k)\Lambda(k)\leq\sum_{k\leq x}\Big(1+\chi\Big(\frac{x}{Mk}\Big)\Big)\Lambda(k)=\psi(x)+\psi(x/M) ,
\end{equation}
and by invoking \eqref{e:sylvester-5} to estimate $\psi(x/M)$, we get
\begin{equation}
\psi(x)\geq V(x)-\psi(x/M)\geq\big(1-\frac{N}{M(N-1)}\big)A(\nu)x+O(\ln^2\!x) .
\end{equation}
\end{proof}

\begin{example}
For the schemes $\nu_4$, $\nu_5$, and $\nu_6$, respectively, we have $M=5$, $M=17$, and $M=13$.
The bound for $\psi$ is now
\begin{equation}
A'(\nu)x+O(\ln^2\!x) \le \psi(x)\le B(\nu)x+O(\ln^2\!x) ,
\end{equation}
with $A'(\nu)=\big(1-\frac{N}{M(N-1)}\big)A(\nu)$ and $B(\nu)=\frac{N}{N-1}A(\nu)$,
and the numerical values for the aforementioned schemes are computed as
\begin{equation}
\begin{aligned}
A'(\nu_4)&=\frac{19}{25}A(\nu_4)=0.7686\ldots,&\qquad B(\nu_4)&=\frac65A(\nu_4)=1.2136\ldots,\\
A'(\nu_5)&=\frac{79}{85}A(\nu_5)=0.8992\ldots,&\qquad B(\nu_5)&=\frac65A(\nu_5)=1.1610\ldots,\\
A'(\nu_6)&=\frac{107}{117}A(\nu_6)=0.8951\ldots,&\qquad B(\nu_6)&=\frac{10}9A(\nu_6)=1.0875\ldots.
\end{aligned}
\end{equation}
Among these $B(\nu_6)$ is found to be better than Chebyshev's $B(\nu_*)=1.1055\ldots$,
indicating that our work was not in vain.
However, the improvement is rather mild.
Sylvester’s iterative procedure, described next, is designed precisely to extract more improvement from such schemes.
\end{example}

\begin{remark}
To continue the theme of Remarks~\ref{r:graph-1}-\ref{r:graph-2}, the bound \eqref{e:sylv-upper} has a nice graphical interpretation.
As an example, Figure \ref{f:e4e} illustrates how to derive the inequalities $V(x)\leq\psi(x)+\psi(x/17)$ and $V(x)\geq\psi(x)-\psi(x/6)$
from (the $E$-function of) the scheme $\nu_5$.
\end{remark}

\begin{figure}[ht]
\centering
\includegraphics[width=.9\textwidth]{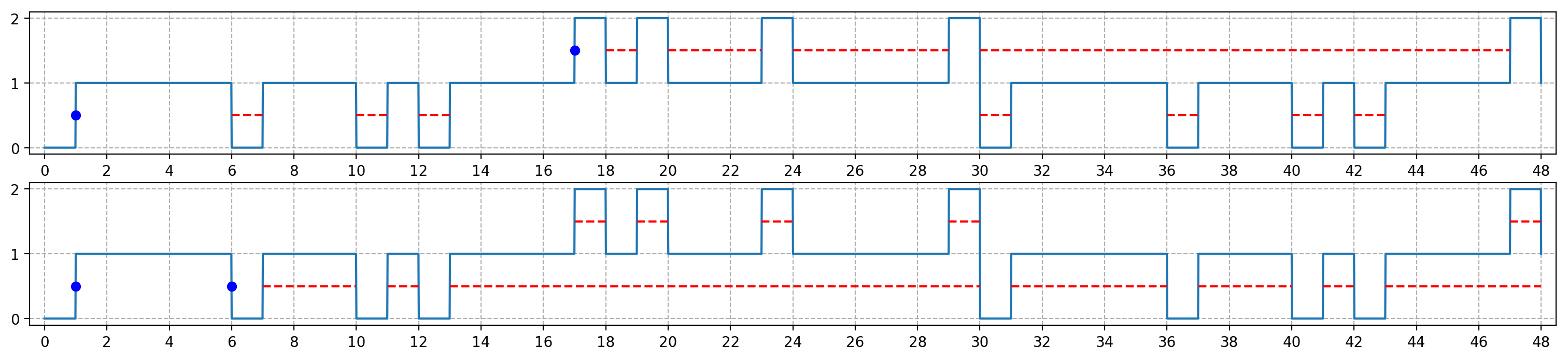}
\caption{Deriving estimates of $\psi$ from $\nu_5$.}
\label{f:e4e}
\end{figure}

\section{Sylvester's iterative refinement}

Sylvester's main contribution to Chebyshev's theory is that he devised a way to iteratively improve the constants of any given scheme.
We explain the procedure for $\nu_4$.
For this scheme, we have
\begin{equation}\label{e:sylvester-6}
a_0x+O(\ln^2\!x)\leq\psi(x)\leq b_0x+O(\ln^2\!x) ,
\end{equation}
where we have introduced the notations $a_0=A'(\nu_4)$ and $b_0=B(\nu_4)$.
To derive these estimates, we used the bounds
\begin{equation}
\begin{split}
V(x)
&=\psi(x)-\psi(x/6)+\underbrace{\psi(x/5)-\psi(x/6)}_{\geq0}+\underbrace{\psi(x/7)+\psi(x/11)-2\psi(x/12)}_{\geq0}+\ldots\\
&\geq\psi(x)-\psi(x/6)
\end{split}
\end{equation}
and
\begin{equation}
\begin{split}
V(x)
&=\psi(x)+\psi(x/5)\underbrace{-2\psi(x/6)+\psi(x/7)+\psi(x/11)}_{\leq0}\underbrace{-2\psi(x/12)+\psi(x/13)+\psi(x/17)}_{\leq0}+\ldots\\
&\leq\psi(x)+\psi(x/5) .
\end{split}
\end{equation}
To improve these bounds, 
we replace the first by 
\begin{equation}\label{e:eg-41b1}
\begin{split}
V(x)
&=\psi(x)-\psi(x/6)+\underbrace{\psi(x/5)-\psi(x/6)}_{\geq0}+\psi(x/7)-\psi(x/12)+\underbrace{\psi(x/11)-\psi(x/12)}_{\geq0}+\ldots\\
&\geq\psi(x)-\psi(x/6)+\psi(x/7)-\psi(x/12)
\end{split}
\end{equation}
and the second by 
\begin{equation}\label{e:eg-41b2}
\begin{split}
V(x)
&=\psi(x)+\psi(x/5)-\psi(x/6)+\psi(x/11)\underbrace{-\psi(x/6)+\psi(x/7)}_{\leq0}\underbrace{-2\psi(x/12)+\ldots}_{\leq0}\\
&\leq\psi(x)+\psi(x/5)-\psi(x/6)+\psi(x/11) .
\end{split}
\end{equation}
The idea is that, for example, in the first estimate, instead of invoking $\psi(x/7)-\psi(x/12)\geq0$ and simply omitting the term $\psi(x/7)-\psi(x/12)$,
it might be beneficial to use the bounds \eqref{e:sylvester-6} to extract some ``positivity'' from $\psi(x/7)-\psi(x/12)$.
The process is depicted in Figure \ref{f:e44}.

\begin{figure}[ht]
\centering
\includegraphics[width=.8\textwidth]{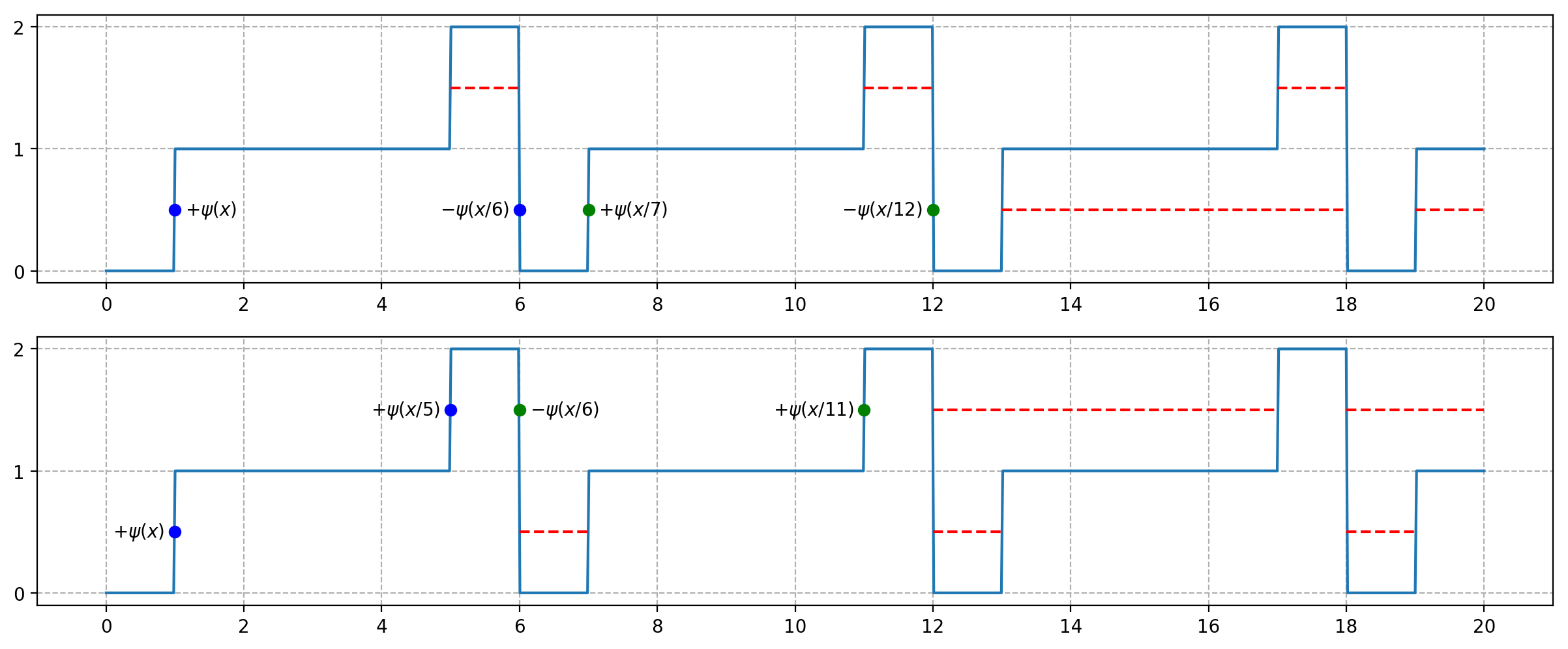}
\caption{Green dots indicate the new terms in \eqref{e:eg-41b1}-\eqref{e:eg-41b2}.}
\label{f:e44}
\end{figure}

On the graph, the blue dots indicate the terms that existed before, and the green dots indicate the newly added terms into the estimate.
In what follows, we will call the terms corresponding to the blue dots the {\em leading terms}.
Thus, by using \eqref{e:sylvester-6} , we get
\begin{equation*}
\begin{split}
\psi(x)-\psi(x/6)&\leq V(x)-\psi(x/7)+\psi(x/12)\leq V(x)-\big(\frac{a_0}7-\frac{b_0}{12}\big)x+O(\ln^2\!x) ,\\
\psi(x)&\geq V(x)-\psi(x/5)+\psi(x/6)-\psi(x/11)\geq V(x)+\big(\frac{a_0}6-\frac{b_0}{11}-\frac{b_0}{5}\big)x+O(\ln^2\!x) ,
\end{split}
\end{equation*}
which leads to
\begin{equation}
\big(A+\frac{a_0}6-\frac{b_0}{11}-\frac{b_0}{5}\big)x+O(\ln^2\!x)\leq\psi(x)\leq \frac65\big(A-\frac{a_0}7+\frac{b_0}{12}\big)x+O(\ln^2\!x) ,
\end{equation}
where $A=A(\nu_4)$.
Hence, with the sequences $\{a_i\}$ and $\{b_i\}$ defined by the recurrence 
\begin{equation}\label{e:sylvester-7}
\begin{cases}
a_{i+1}&=A+\frac{a_i}6-\frac{b_i}{11}-\frac{b_i}{5} ,\\
b_{i+1}&=\frac65\big(A-\frac{a_i}7+\frac{b_i}{12}\big) ,
\end{cases}
\end{equation}
we have the bounds
\begin{equation}
a_ix+O(\ln^2\!x)\leq\psi(x)\leq b_ix+O(\ln^2\!x) .
\end{equation}
Now, assuming that $a_i\to a=\alpha A$ and $b_i\to b=\beta A$ as $i\to\infty$, 
in the limit,
the aforementioned recurrence becomes 
\begin{equation}
\begin{cases}\alpha&=1+\frac{\alpha}6-\frac{\beta}{11}-\frac{\beta}{5}\\\beta&=\frac65\big(1-\frac{\alpha}7+\frac{\beta}{12}\big)\end{cases} ,
\end{equation}
yielding
\begin{equation}
\alpha=\frac{4242}{5391}=0.7868\ldots,\qquad\beta=\frac{6380}{5391}=1.18345\ldots .
\end{equation}
So if our assumption 
\begin{equation}\label{e:eg-41}
a_i\to a=\alpha A=0.7958\ldots,\qquad b_i\to b=\beta A=1.1969\ldots
\end{equation}
holds, then for any $\varepsilon>0$, we have
\begin{equation}
(a-\varepsilon)x+O(\ln^2\!x)\leq\psi(x)\leq (b+\varepsilon)x+O(\ln^2\!x) ,
\end{equation}
meaning that the constants $a$ and $b$ obtained by the recurrence are an improvement on the original constants 
\begin{equation}
A'(\nu_4)=0.7686\ldots,\qquad B(\nu_4)=1.2136\ldots .
\end{equation}
To verify the assumption, we let $a_i=(\alpha+\Delta\alpha_i)A$ and $b_i=(\beta+\Delta\beta_i)A$,
and write the recurrence as
\begin{equation}
\begin{cases}\Delta\alpha_{i+1}&=\frac{\Delta\alpha_i}6-\frac{\Delta\beta_i}{11}-\frac{\Delta\beta_i}{5}=\frac{\Delta\alpha_i}6-\frac{16}{55}\Delta\beta_i\\\Delta\beta_{i+1}&=\frac65\big(-\frac{\Delta\alpha_i}7+\frac{\Delta\beta_i}{12}\big)=-\frac{6}{35}\Delta\alpha_i+\frac1{10}{\Delta\beta_i}\end{cases}
\end{equation}
or
\begin{equation}
\begin{pmatrix}\Delta\alpha_{i+1}\\\Delta\beta_{i+1}\end{pmatrix}=\begin{pmatrix}\frac16&-\frac{16}{55}\\-\frac{6}{35}&\frac1{10}\end{pmatrix}\begin{pmatrix}\Delta\alpha_{i}\\\Delta\beta_{i}\end{pmatrix} .
\end{equation}
The eigenvalues of the involved matrix are
\begin{equation}
\lambda_1=-0.0924\ldots,\qquad\lambda_2=0.3591\ldots ,
\end{equation}
and therefore regardless of the values of $C_1,\ldots,C_4$, the sequences 
\begin{equation}
\Delta\alpha_{i}=C_{1}\lambda_1^i+C_{2}\lambda_2^i,\qquad\Delta\beta_{i}=C_{3}\lambda_1^i+C_{4}\lambda_2^i ,
\end{equation}
both converge to $0$ as $i\to\infty$.
Moreover, the outcome of the procedure, the constants $a$ and $b$ do not depend on the initial values $a_0$ and $b_0$,
and depends only on the coefficients of the recurrence \eqref{e:sylvester-7}.
Indeed, the preceding argument is valid even when $a_0=0$ or $a_0<0$.
Sylvester's procedure gives better and better bounds as it iterates, starting with, say, $a_0=0$ and $b_0=100$.
In Figure \ref{f:iter1}, we depict the iterations corresponding to two initial value pairs $a_0=0.1$, $b_0=3$ and $a_0=0$, $b_0=6$.

\begin{figure}[ht]
\centering
\includegraphics[width=.7\textwidth]{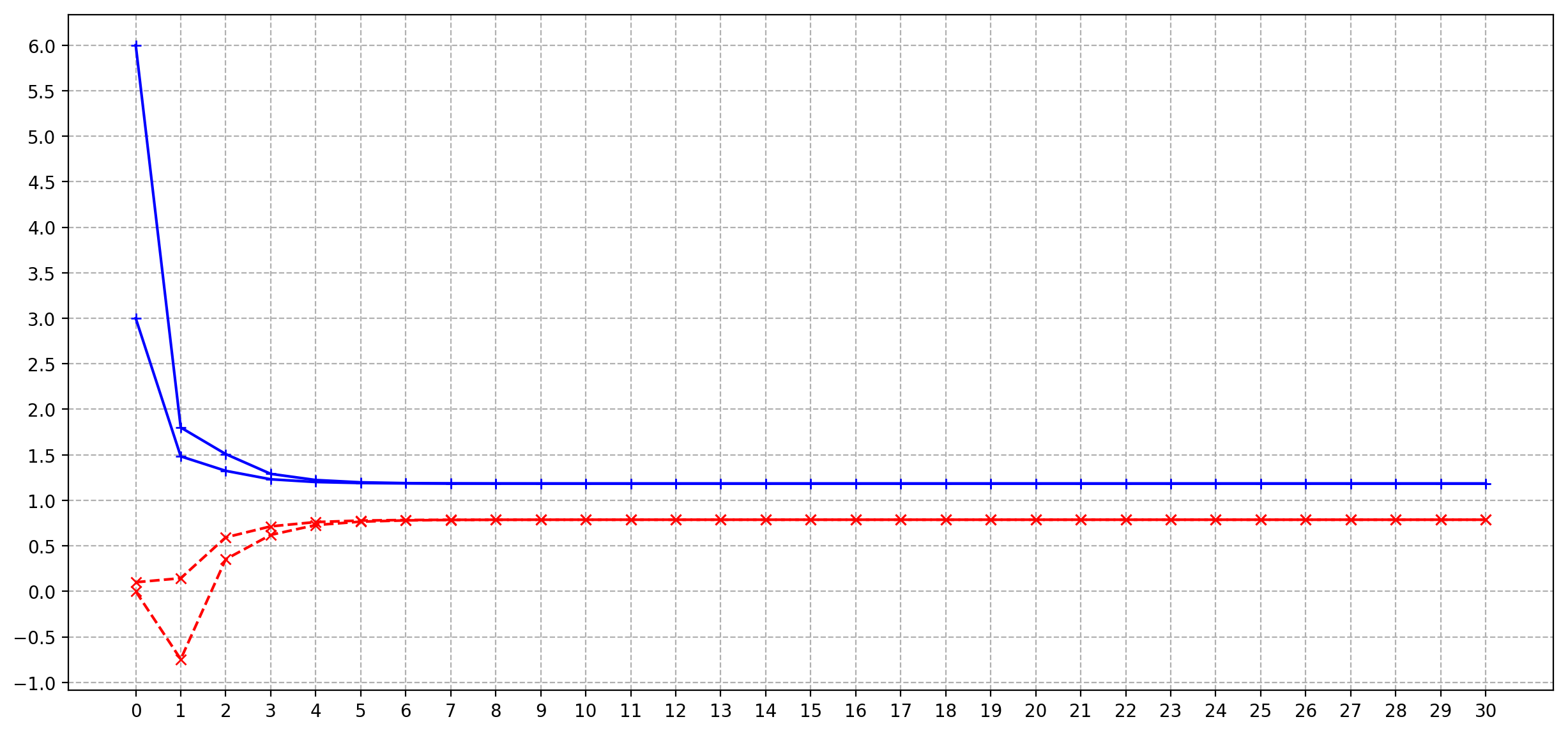}
\caption{Sequences $a_i$ (red) and $b_i$ (blue).}
\label{f:iter1}
\end{figure}

\section{Computational optimization}

Before applying Sylvester's procedure to other schemes, let us consider the question 
of how one could choose the terms of the expansion of $V(x)$ to keep in the inequalities that lead to the recurrence relation.
In the $i$-th step of the iteration, we would have the bounds
\begin{equation}
a_ix+O(\ln^2\!x)\leq\psi(x)\leq b_ix+O(\ln^2\!x) .
\end{equation}
In the next step, suppose that we include the term $\psi(x/m)-\psi(x/n)$.
Since
\begin{equation}
V(x)=\psi(x)+\ldots+\psi(x/m)-\psi(x/n)+\ldots
\displaystyle\geq\psi(x)+\ldots+\big(\frac{a_i}{m}-\frac{b_i}{n}\big)x+\ldots ,
\end{equation}
the upper bound would improve if
$\displaystyle\frac{a_i}{m}-\frac{b_i}{n}>0$ or
\begin{equation}
\frac{n}{m}>\frac{b_i}{a_i} ,
\end{equation}
and the upper bound would get worse if
\begin{equation}\label{e:eg-4rat}
\frac{n}{m}<\frac{b_i}{a_i} .
\end{equation}
In other words, we shouldn't include pairs whose indices are too close together.
Obviously, the same applies to lower bounds as well.

Thus, in order to automate the process, 
we fix a constant $\rho>1$, 
and apart from the leading terms,
let us include 
\begin{itemize}
\item
the pairs $\psi(x/m)-\psi(x/n)$ satisfying $\frac{n}{m}>\rho$ into the upper bound of $V(x)$,
\item
and the pairs $-\psi(x/m)+\psi(x/n)$ satisfying $\frac{n}{m}>\rho$ into the lower bound of $V(x)$.
\end{itemize}
As an example, Figure \ref{f:e444} shows the new terms in green for the scheme $\nu_4$, when $\rho=1.3$.

\begin{figure}[ht]
\centering
\includegraphics[width=.8\textwidth]{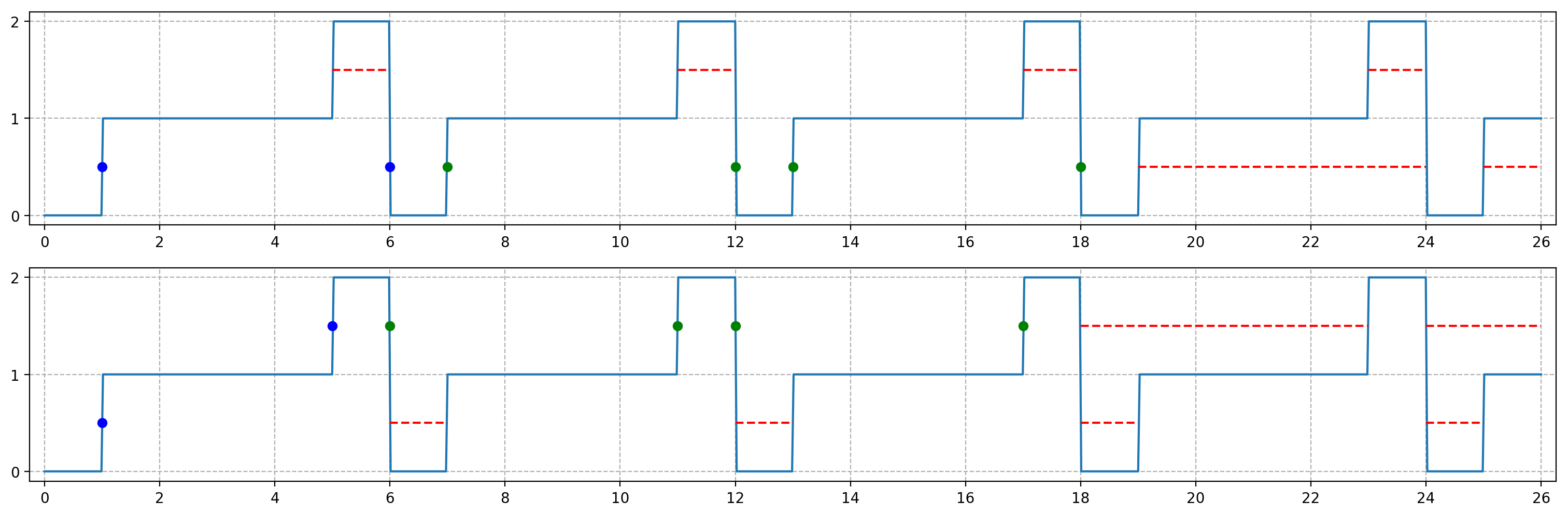}
\caption{With $\rho=1.3$, we get one additional pair in each bound, cf. Figure \ref{f:e44}.}
\label{f:e444}
\end{figure}

Compared to the previous example \eqref{e:eg-41b1}-\eqref{e:eg-41b2},
each bound has now one additional pair.
Namely, we get the bounds
\begin{equation}
\begin{split}
V(x)&\geq\psi(x)-\psi(x/6)+\psi(x/7)-\psi(x/12)+\psi(x/13)-\psi(x/18)\\
V(x)&\leq\psi(x)+\psi(x/5)-\psi(x/6)+\psi(x/11)-\psi(x/12)+\psi(x/17)
\end{split}
\end{equation}
leading to the iteration
\begin{equation}
\begin{cases}
a_{i+1}&=A+\big(\frac16+\frac1{12}\big)a_i-\big(\frac15+\frac1{11}+\frac1{17}\big)b_i\\
b_{i+1}&=\frac65A-\big(\frac17+\frac1{13}\big)\cdot\frac65a_i+\big(\frac1{12}+\frac1{18}\big)\cdot\frac65b_i
\end{cases}
\end{equation}
whose limit behaviour is
\begin{equation}
a_i\to a=0.7852\ldots,\qquad b_i\to b=1.5381\ldots .
\end{equation}
This is weaker than the values
$a=0.7958\ldots$ and $b=1.1969\ldots$ that we found earlier in \eqref{e:eg-41},
meaning that we have added too many new terms into the bounds.
The additional terms we have added were $\psi(x/13)-\psi(x/18)$ and $-\psi(x/12)+\psi(x/17)$.
The index ratios of these pairs, respectively, are 
\begin{equation}
\frac{18}{13}=1.3846\ldots<\frac{b}{a}=1.5381\ldots,\qquad\frac{17}{12}=1.4166\ldots<\frac{b}{a} ,
\end{equation}
which, in light of \eqref{e:eg-4rat}, indicate that 
towards the end of the iteration, those pairs were in fact harming the ratio $\frac{b}{a}$, rather than helping.
To conclude 
\begin{itemize}
\item
If the index ratio $\frac{n}m$ of a given pair $\pm\psi(x/m)\mp\psi(x/n)$ is less than the final ratio $\frac{b}{a}$, 
then by removing the pair from the procedure, the final ratio will be improved.
\item
Hence we should try to choose $\rho$ so that $\rho\approx\frac{b}{a}$.
\end{itemize}
In this section we take an experimental point of view and systematically
explore the effect of varying the parameter $\rho$ and the underlying
scheme $\nu$, using numerical computation to guide the choice of
constants and to identify near-optimal configurations.

In Figure \ref{f:e4ab},
we illustrate how the outcome of Sylvester's procedure for $\nu_4$ depends on the parameter $\rho$.
The first graph shows the constants $a$ and $b$,
the second graph the eigenvalues of the iteration matrix,
and the third graph the number of terms included into the upper and lower bounds of $V(x)$.
The plots suggest that near $\rho=1.5$ the values of $a$ and $b$ are essentially optimal within this one parameter family.
Near this value, the upper and lower bounds of $V(x)$ each have four terms.
At $\rho=1.5$, we have $a=0.7958\ldots$ and $b=1.1969\ldots$, which we found earlier in \eqref{e:eg-41}.

\begin{figure}[ht]
\centering
\includegraphics[width=.9\textwidth]{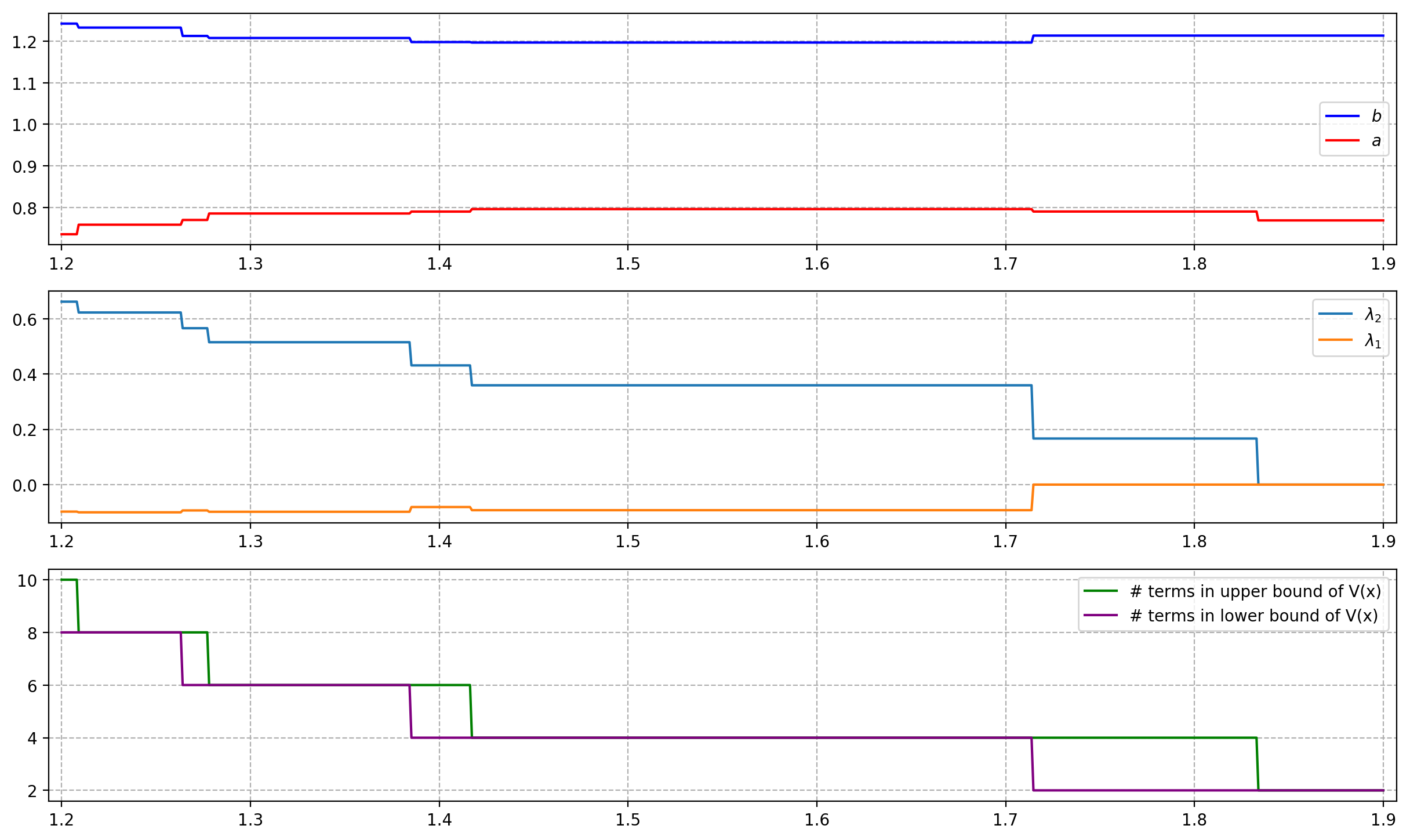}
\caption{$\nu_4$.}
\label{f:e4ab}
\end{figure}

Now, let us apply Sylvester's procedure to Chebyshev's scheme.
Taking $\rho=1.2$, we get
\begin{equation}
\begin{split}
V(x)&=\psi(x)-\psi(x/6)+\psi(x/7)-\psi(x/10)+\underbrace{\psi(x/11)-\psi(x/12)+\ldots}_{\geq0}\\&\quad\geq\psi(x)-\psi(x/6)+\psi(x/7)-\psi(x/10) ,\\
V(x)&=\psi(x)\underbrace{-\psi(x/6)+\psi(x/7)-\ldots+\psi(x/23}_{\leq0})-\psi(x/24)+\psi(x/29)\underbrace{-\psi(x/30)+\ldots}_{\leq0}\\&\quad\leq\psi(x)-\psi(x/24)+\psi(x/29) ,
\end{split}
\end{equation}
cf. Figure \ref{f:echebcheb}.

\begin{figure}[h!t]
\centering
\includegraphics[width=.9\textwidth]{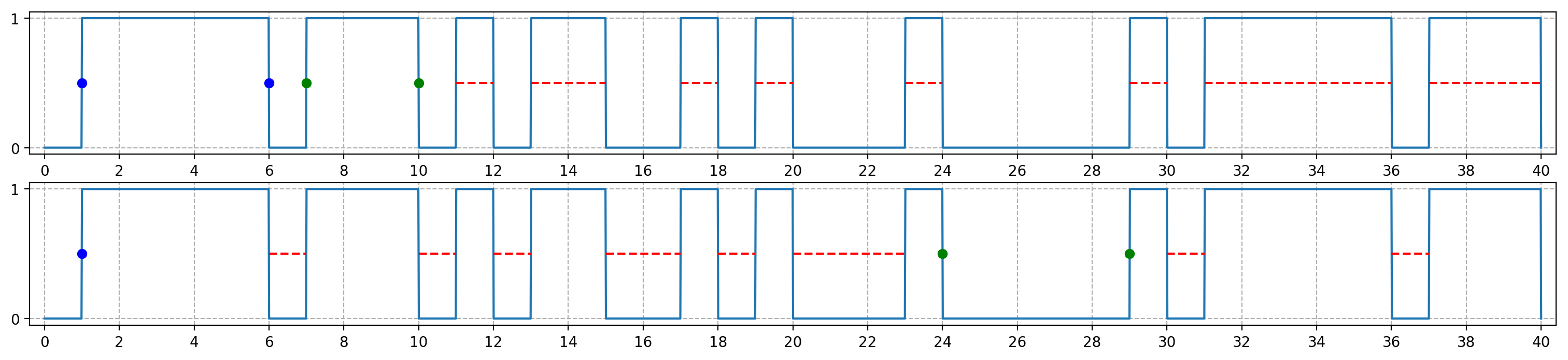}
\caption{Chebyshev's scheme, with $\rho=1.2$.}
\label{f:echebcheb}
\end{figure}

The corresponding iteration is
\begin{equation}
\begin{cases}a_{i+1}&=A+\frac{a_i}{24}-\frac{b_i}{29}\\b_{i+1}&=\frac65\big(A-\frac{a_i}7+\frac{b_i}{10}\big)\end{cases}
\end{equation}
where $A=A(\nu_*)=0.9212\ldots$.
Proceeding as before, we set $a_i=\alpha_i A$ and $b_i=\beta_i A$, and write the iteration as
\begin{equation}
\begin{pmatrix}\alpha_{i+1}\\\beta_{i+1}\end{pmatrix}=\begin{pmatrix}1\\\frac65\end{pmatrix}+\begin{pmatrix}\frac1{24}&-\frac1{29}\\-\frac{6}{35}&\frac6{50}\end{pmatrix}\begin{pmatrix}\alpha_{i}\\\beta_{i}\end{pmatrix} .
\end{equation}
Since the eigenvalues of the matrix are
\begin{equation}
\lambda_1=-0.0054\ldots,\qquad\lambda_2=0.1671\ldots ,
\end{equation}
we conclude that no matter the values of the constants $C_1,\ldots,C_4$,
the sequences
\begin{equation}
\alpha_{i}=\alpha+C_{1}\lambda_1^i+C_{2}\lambda_2^i,\qquad\beta_{i}=\beta+C_{3}\lambda_1^i+C_{4}\lambda_2^i
\end{equation}
converge to $\alpha$ and $\beta$ respectively, as $i\to\infty$.
Those limits can easily be found by putting $\alpha_i=\alpha$ and $\beta_i=\beta$ into the iteration, as
\begin{equation}
\alpha=\frac{51072}{50999},\qquad\beta=\frac{59595}{50999} ,
\end{equation}
which implies that the inequalities
\begin{equation}
a_ix+O(\ln^2\!x)\leq\psi(x)\leq b_ix+O(\ln^2\!x)
\end{equation}
are true, with
\begin{equation}\label{e:eg-chebab}
a_i\to\alpha A=0.9226\ldots,\qquad b_i\to\beta A=1.0765\ldots .
\end{equation}
This is a significant improvement over Chebyshev's original 
\begin{equation}
A(\nu_*)=A=0.9212\ldots,\qquad B(\nu_*)=\frac65A=1.1055\ldots
\end{equation}
In fact, the values in \eqref{e:eg-chebab} are essentially optimal for Chebyshev's scheme,
cf. Figure \ref{f:ecab}.

\begin{figure}[ht]
\centering
\includegraphics[width=.9\textwidth]{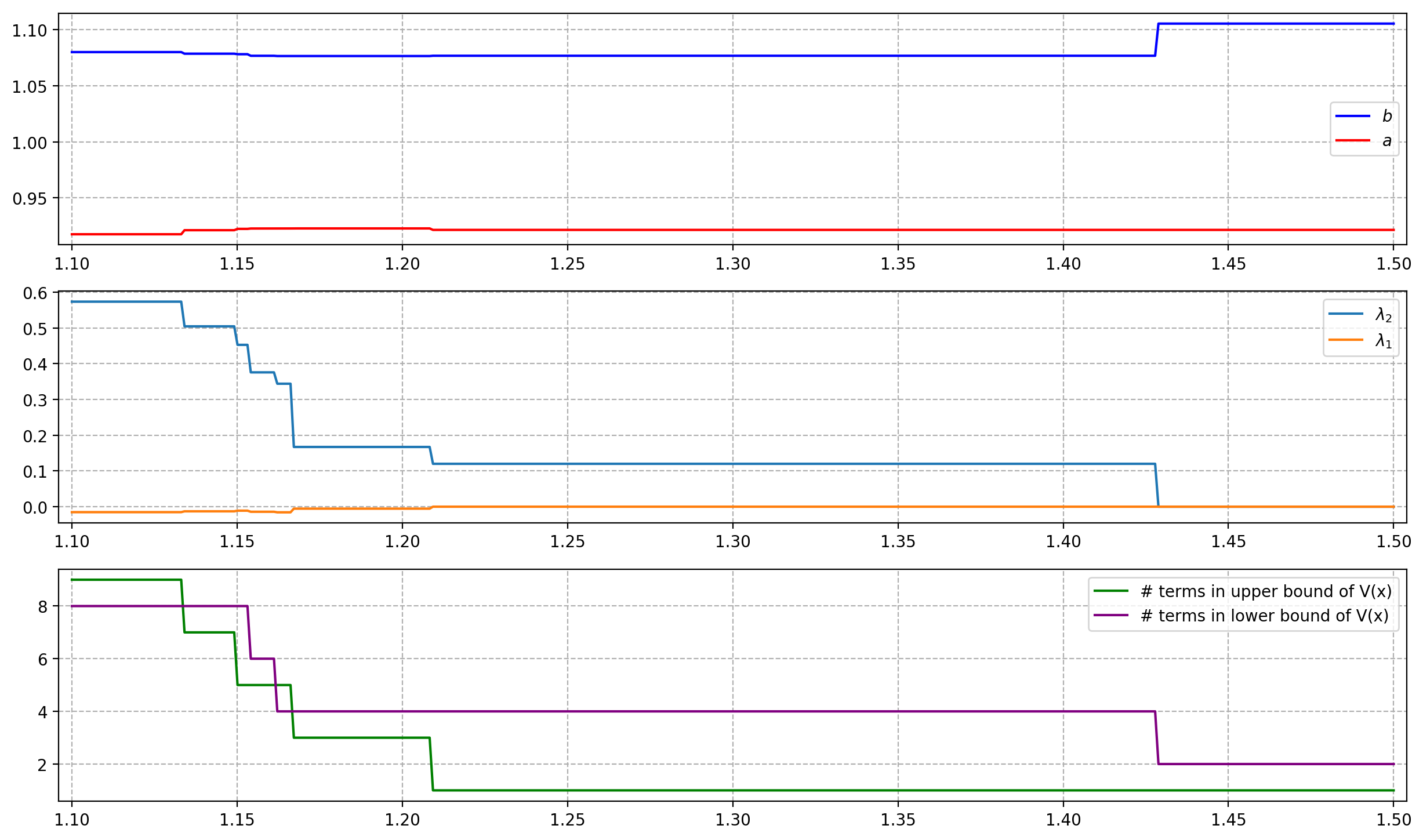}
\caption{Chebyshev's scheme, with varying $\rho$.}
\label{f:ecab}
\end{figure}

The values in \eqref{e:eg-chebab} are obtained by Sylvester in his 1881 work.
In the same work, the schemes $\nu_1$ and $\nu_4$ appeared for the first time. 
Sylvester published another paper on the subject in 1892, where he presents, apart from $\nu_2$, $\nu_3$, $\nu_5$, and $\nu_6$,
the following gigantic schemes
\begin{equation}
\begin{split}
\nu_7&=[1,6,10,210,231,1155;2,3,5,7,11,105] , \\
\nu_8&=[1, 6, 10, 14, 105;2, 3, 5, 7, 11, 13, 385, 1001] .
\end{split}
\end{equation}

By applying Sylvester's procedure to $\nu_5$ with $\rho=1.2$, we get the bounds
\begin{equation}
\begin{split}
V(x)
&\leq\psi(x)+\psi(x/17)-\psi(x/24)+\psi(x/29)\\
&\quad-\psi(x/30)+\psi(x/47)-\psi(x/60)+\psi(x/77) ,\\
V(x)
&\geq\psi(x)-\psi(x/6)+\psi(x/7)-\psi(x/10)+\psi(x/13)-\psi(x/30)\\
&\quad+\psi(x/43)-\psi(x/60)+\psi(x/73)-\psi(x/90) ,
\end{split}
\end{equation}
leading to the iteration
\begin{equation}
\begin{cases}
a_{i+1}&=A+\big(\frac1{24}+\frac1{30}+\frac1{60}\big)a_i-\big(\frac1{17}+\frac1{29}+\frac1{47}+\frac1{77}\big)b_i ,\\
b_{i+1}&=\frac65A-\big(\frac17+\frac1{13}+\frac1{43}+\frac1{73}\big)\cdot\frac65a_i+\big(\frac1{10}+\frac1{30}+\frac1{60}+\frac1{90}\big)\cdot\frac65b_i ,
\end{cases}
\end{equation}
cf. Figure \ref{f:e5e5}.

\begin{figure}[ht]
\centering
\includegraphics[width=.9\textwidth]{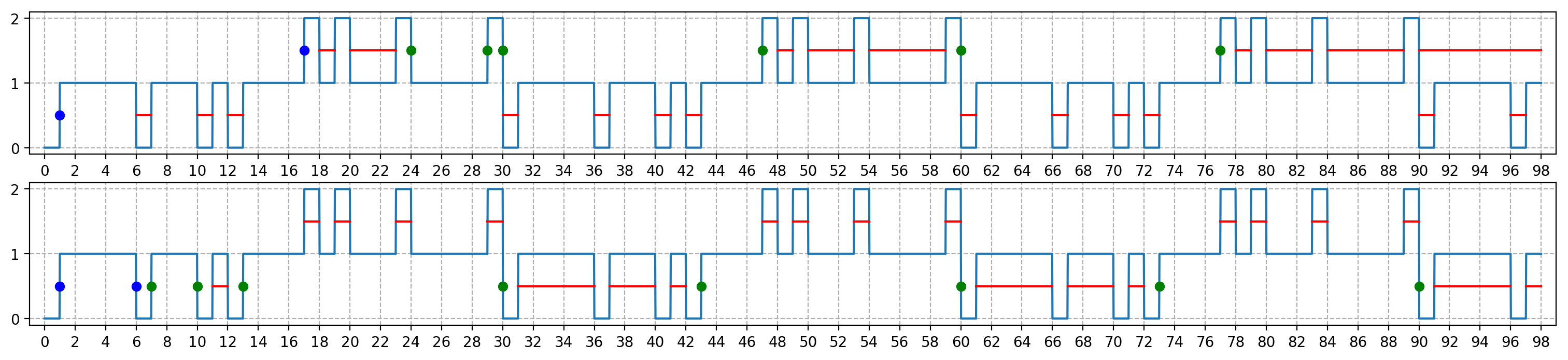}
\caption{Scheme $\nu_5$ with $\rho=1.2$.}
\label{f:e5e5}
\end{figure}

The limit of the iteration is found to be
\begin{equation}
a_i\to a=0.9119\ldots,\qquad b_i\to b=1.0909\ldots ,
\end{equation}
whose ratio is $\frac{b}{a}=1.1963\ldots$.
These values are virtually optimal for $\nu_5$, cf. Figure \ref{f:e5ab}.

\begin{figure}[ht]
\centering
\includegraphics[width=.9\textwidth]{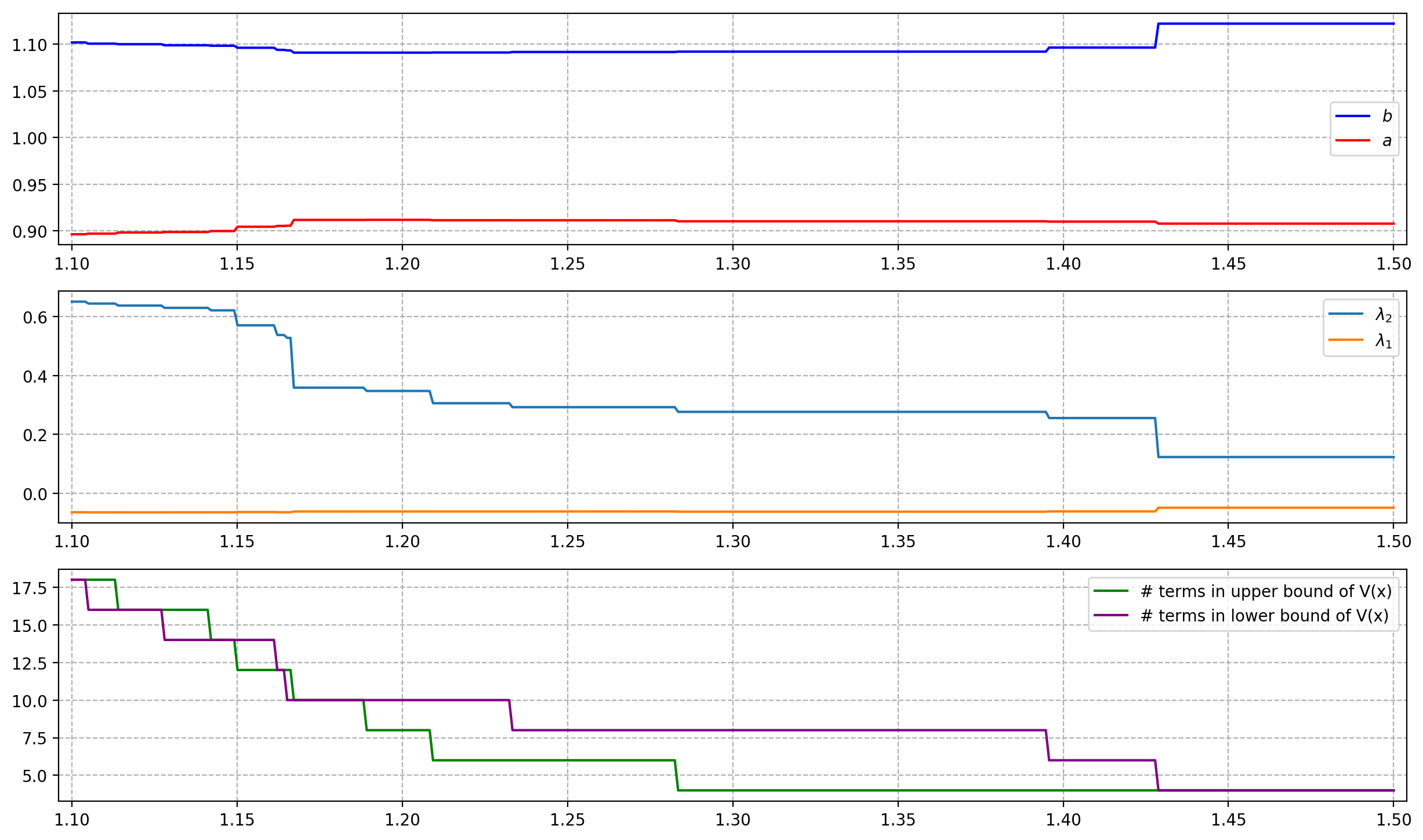}
\caption{Scheme $\nu_5$ with varying $\rho$.}
\label{f:e5ab}
\end{figure}

Next, for $\nu_6$, we take $\rho=1.1$, and get the bounds
\begin{equation}
\begin{split}
V(x)&\leq\psi(x)+\psi(x/13)-\psi(x/10)+\psi(x/11)-\psi(x/15)+\psi(x/17)\\
&\quad-\psi(x/14)+\psi(x/19)-\psi(x/20)+\psi(x/73)-\psi(x/230)+\psi(x/283)\\
&\quad-\psi(x/440)+\psi(x/493)-\psi(x/110)+\psi(x/139) , \\
V(x)&\geq\psi(x)-\psi(x/10)+\psi(x/11)-\psi(x/15)+\psi(x/17)-\psi(x/21)\\
&\quad+\psi(x/23)-\psi(x/28)+\psi(x/31)-\psi(x/35)+\psi(x/71)-\psi(x/100)\\
&\quad+\psi(x/281)-\psi(x/310)+\psi(x/137)-\psi(x/190)+\psi(x/347)-\psi(x/400) .
\end{split}
\end{equation}
Sylvester himself took $\rho=1.1$, but omitted (for some reason) the terms $-\psi(x/440)+\psi(x/493)$ and $\psi(x/281)-\psi(x/310)$ respectively,
from the afore-displayed bounds,
to get
\begin{equation}
a=0.941854\ldots,\qquad b=1.056726\ldots .
\end{equation}
If we include the omitted terms, we get
\begin{equation}
a=0.941806\ldots,\qquad b=1.056825\ldots ,
\end{equation}
suggesting that perhaps Sylvester intentionally omitted those terms.
In any case, none of the aforementioned values are optimal for $\nu_6$,
as the choice $\rho=1.105$ gives
\begin{equation}
a=0.944462\ldots,\qquad b=1.055800\ldots .
\end{equation}
Figure \ref{f:e6ab} suggests that these values are the best possible.

\begin{figure}[ht]
\centering
\includegraphics[width=.9\textwidth]{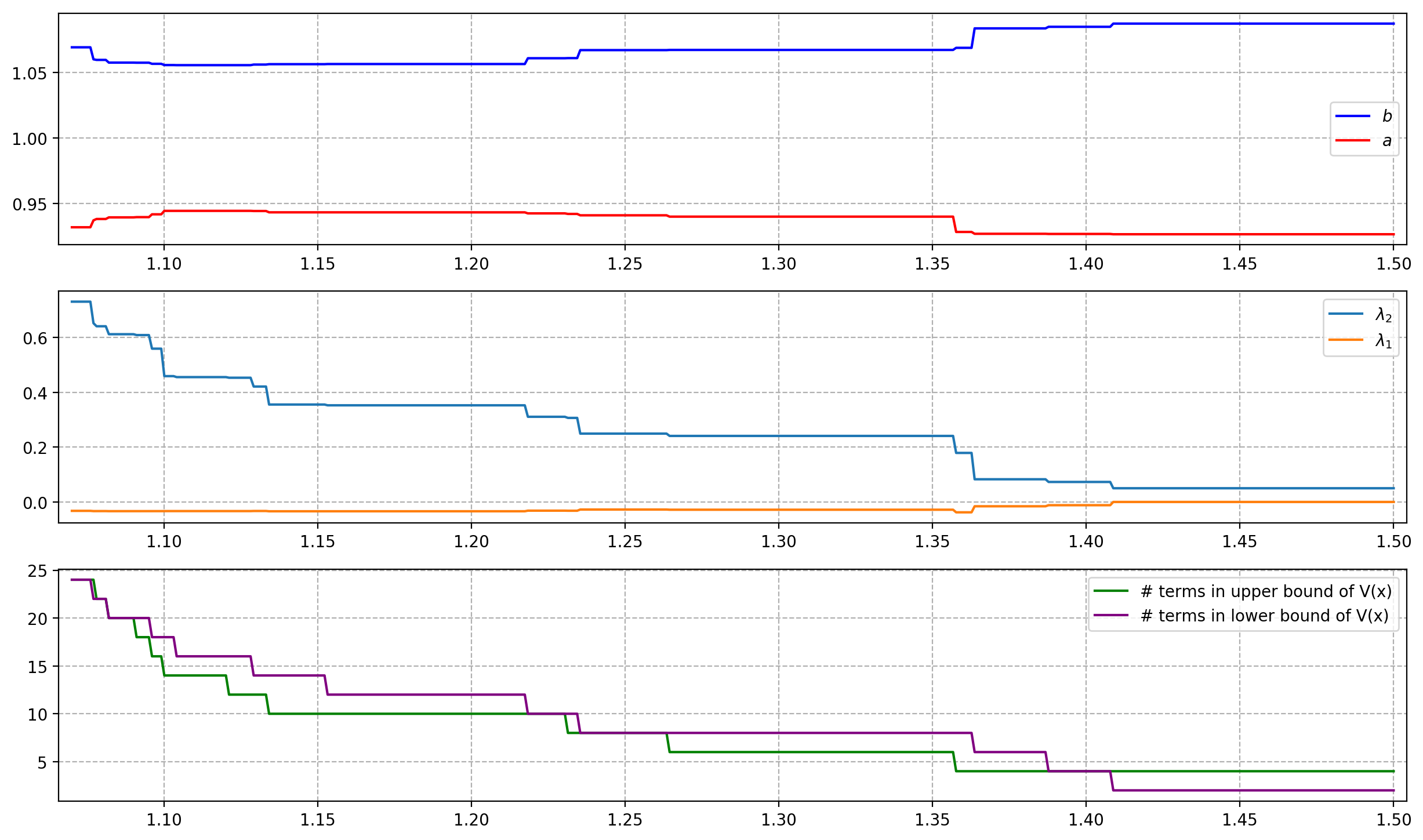}
\caption{Scheme $\nu_6$ with varying $\rho$.}
\label{f:e6ab}
\end{figure}

\section{Sylvester's gigantic schemes}

Finally, let us consider the two gigantic schemes. 
For $\nu_7$, the period of $E(x)$ is 2310, and we have $-2\leq E(x)\leq2$ for all $x$. 
Moreover, we have $E(x)\geq1$ for $x<15$,
and the first occurrence of $E(x)=2$ is at $x=13$.
The first occurrence of $E(x)=-1$ is at $x=105$,
and first occurrence of $E(x)=-2$ is at $x=616$, cf. Figure \ref{f:e7}.

\begin{figure}[ht]
\centering
\includegraphics[width=.9\textwidth]{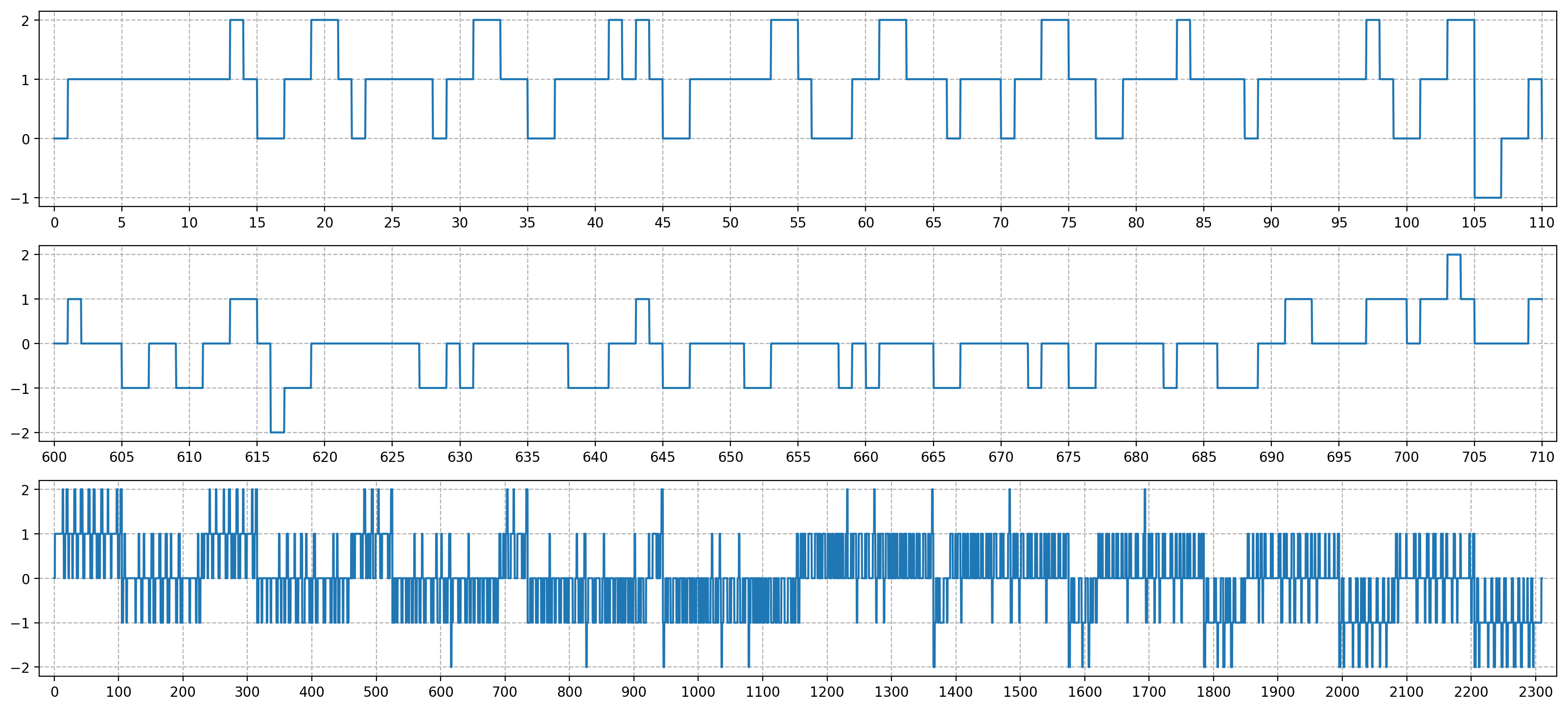}
\caption{$E(x)$ for $\nu_7$.}
\label{f:e7}
\end{figure}

Since $E$ takes negative values for these schemes,
Theorem~\ref{t:cheb}(c) does not apply to them. The estimates below
are instead obtained directly from the paired $\psi$-expansions in
Sylvester's iterative procedure, using already established bounds as
starting bounds.

Especially the lower bound of $V(x)$ is quite difficult to manage, and hence
Sylvester took the upper bound generated by $\nu_7$ with $\rho=1.1$,
and combined it with the lower bound coming from  $\nu_6$,
to obtain
\begin{equation}
a=0.946197\ldots,\qquad b=1.055185\ldots .
\end{equation}
After this, he applied the lower bound coming from $\nu_7$, 
but only including the terms up to $\psi(x/616)$, which improved $b$ to
\begin{equation}
b=1.054239\ldots .
\end{equation}
Since we have computer at our disposal,
we easily find that $\rho=1.113$ is essentially the optimal parameter for $\nu_7$, for which
\begin{equation}
a=0.946585\ldots,\qquad b=1.054309\ldots ,
\end{equation}
cf. Figure \ref{f:e7ab}.
Note that by hybridizing $\nu_7$ with $\nu_6$, Sylvester made the value of $b$ better than that yielded by $\nu_7$ alone. 

\begin{figure}[ht]
\centering
\includegraphics[width=.9\textwidth]{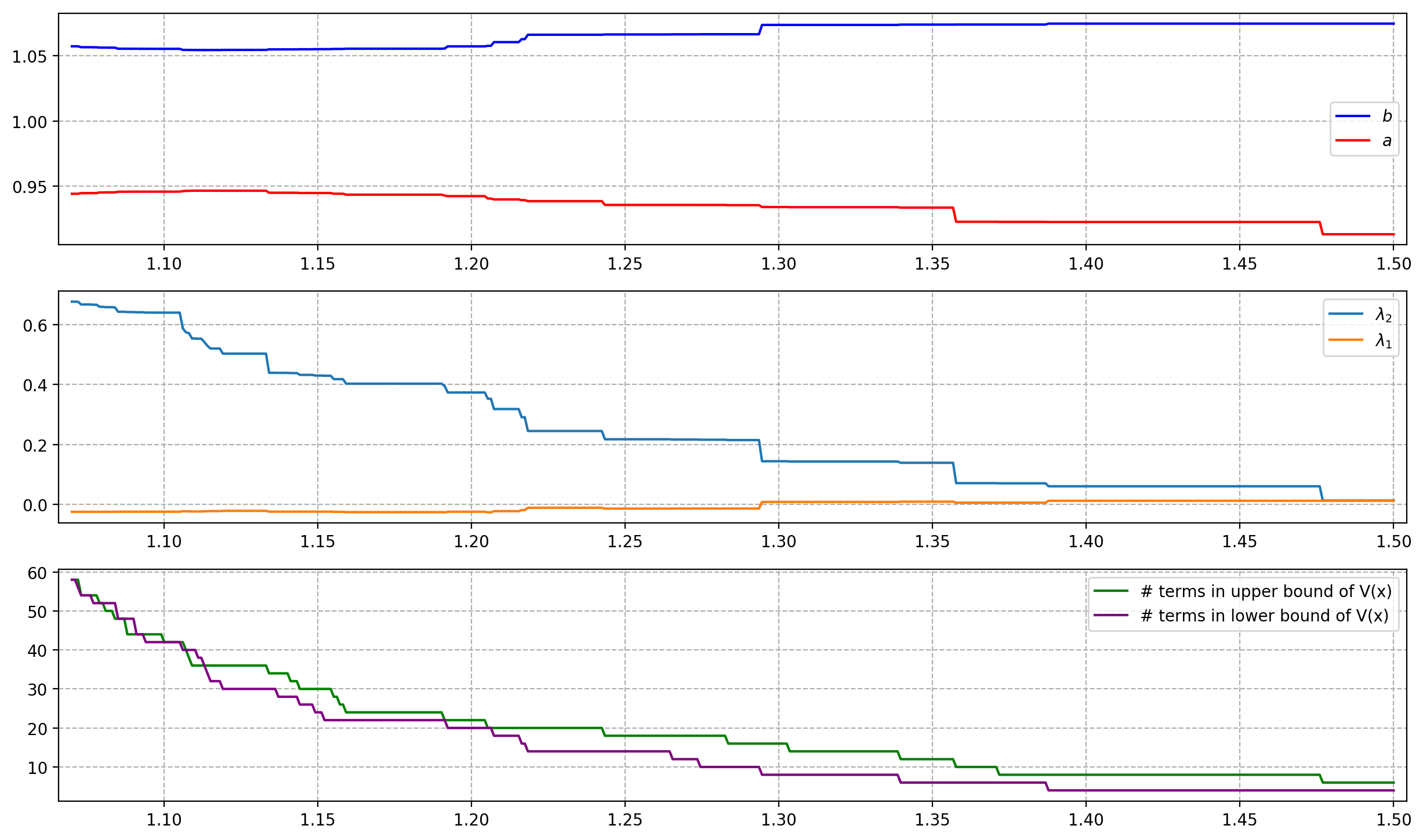}
\caption{Scheme $\nu_7$ with varying $\rho$.}
\label{f:e7ab}
\end{figure}

It is now time to look at our absolute champion $\nu_8$.
For this scheme, the period of $E(x)$ is 30030, and we have $-1\leq E(x)\leq4$ for all $x$. 
Furthermore, we have $E(x)\geq1$ for $x<15$.
The first occurrence of $E(x)=2$ is at $x=19$,
the first occurrence of $E(x)=-1$ is at $x=66$,
the first occurrence of $E(x)=3$ is at $x=229$,
and the first occurrence of $E(x)=4$ is at $x=1891$.

\begin{figure}[ht]
\centering
\includegraphics[width=.9\textwidth]{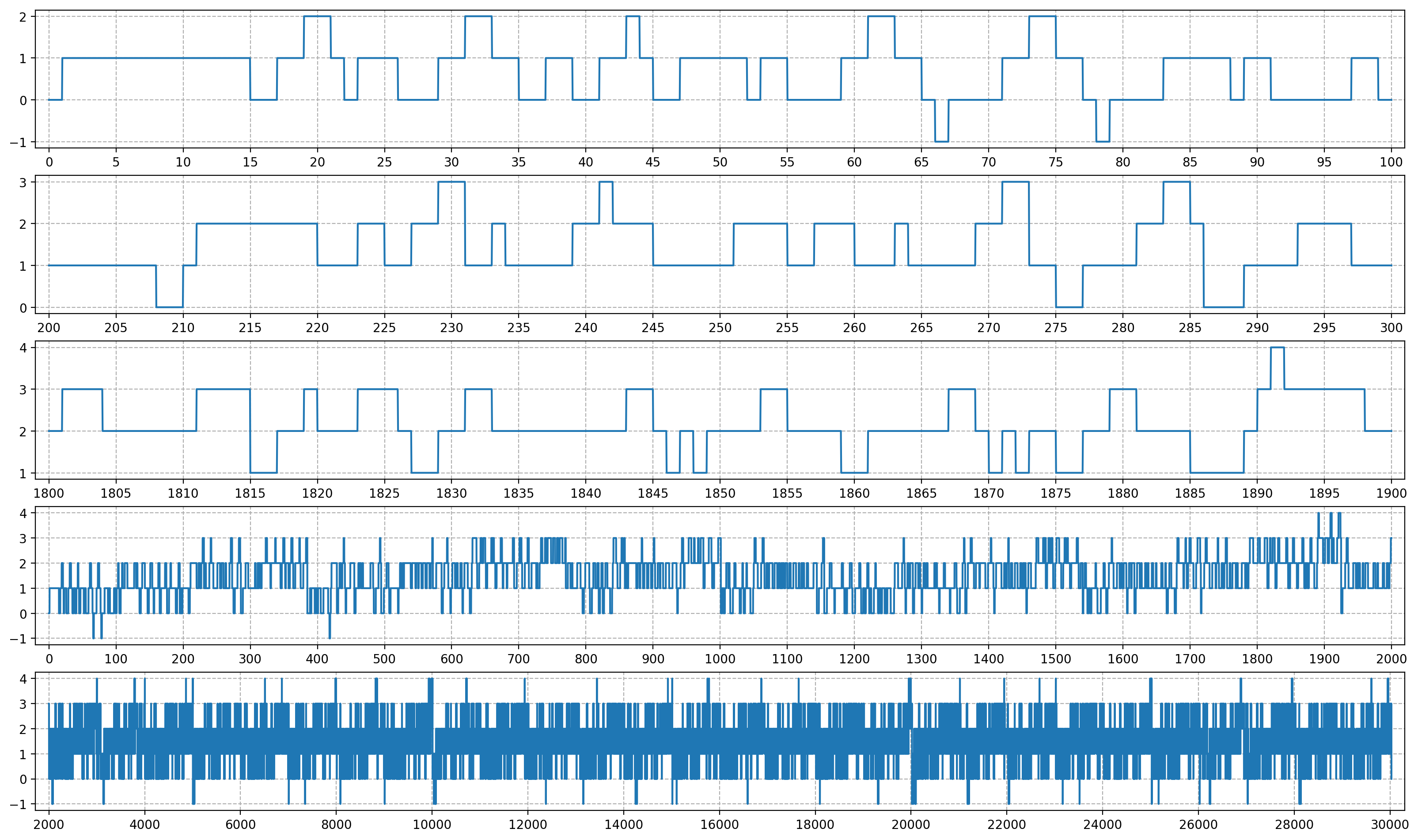}
\caption{Scheme $\nu_8$.}
\label{f:e8}
\end{figure}

Sylvester used $\nu_8$ to obtain his personal record
\begin{equation}
a=0.95695\ldots,\qquad b=1.04423\ldots .
\end{equation}
Nevertheless, we find that the value $\rho=1.09$ is closer to optimality,
yielding
\begin{equation}
a=0.957600\ldots,\qquad b=1.043521\ldots ,
\end{equation}
cf. Figure \ref{f:e8ab}.

\begin{figure}[ht]
\centering
\includegraphics[width=.9\textwidth]{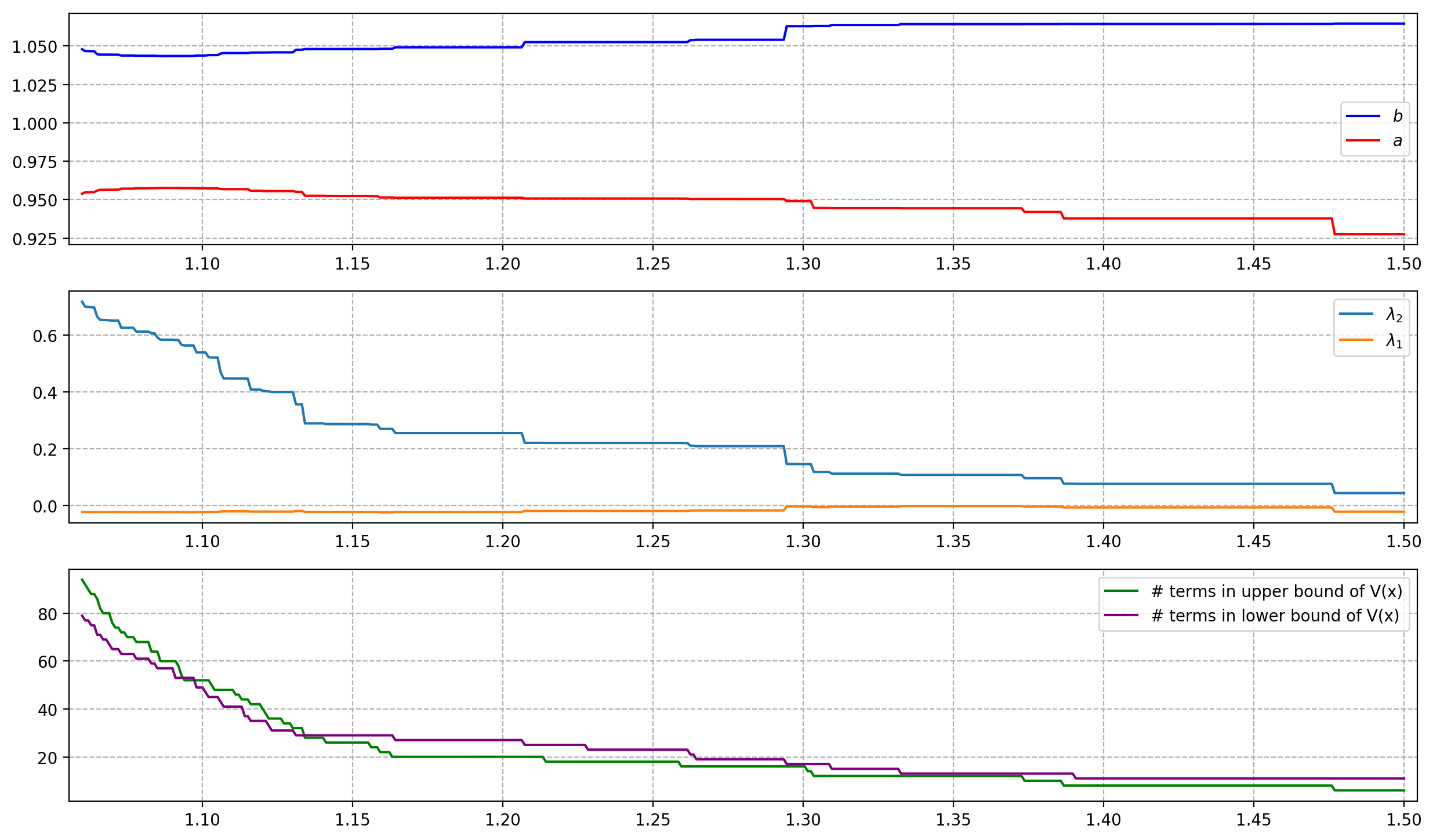}
\caption{Scheme $\nu_8$ with varying $\rho$.}
\label{f:e8ab}
\end{figure}

\section{Conclusion}
The Chebyshev-Sylvester method is a powerful example of elementary reasoning in number theory. 
It begins with a fundamental identity relating the summatory logarithm function $T(x)$ to the Chebyshev function $\psi(x)$. 
By replacing the M\"obius function with a finitely-supported proxy $\nu$, one can derive initial, explicit bounds on $\psi(x)$. Sylvester's main contribution was an iterative procedure that uses these initial bounds to bootstrap a sequence of progressively better bounds, converging to the fixed-point values associated with the chosen pairing of terms.

While superseded by complex-analytic methods for proving the Prime Number Theorem, the method remains of great pedagogical value. 
It provides a concrete case study in the techniques of elementary analytic number theory, including convolution identities, summatory functions, and bootstrapping arguments. 
The computational optimization presented here makes this classical theory accessible and demonstrates the full power of Sylvester's ingenious machinery.

It may be helpful to place the present work in the context of some
standard references.
Davenport's multiplicative number theory text \cite{DavenportMNT}
presents the Chebyshev functions and the von Mangoldt function
primarily as entry points to analytic methods via the
Dirichlet series 
\[
-\frac{\zeta'(s)}{\zeta(s)} = \sum_{n\ge1} \frac{\Lambda(n)}{n^s} .
\]
Diamond's survey \cite{DiamondElem} collects a wide range of
elementary arguments in prime number theory, including Chebyshev's
original inequalities, Erd\H{o}s--Selberg's elementary proof of the
prime number theorem, and related refinements.
In contrast, we have concentrated here on a single elementary
mechanism --- the use of finitely supported schemes $\nu$,
the associated function $E(x)$, and the Sylvester iteration --- and
have made this framework fully explicit and computational, without
calling on complex analysis or the more global themes.

In a different but closely related direction, Diamond and Erd\H{o}s
\cite{DiamondErdos1980} showed that Chebyshev-type elementary
arguments can in principle be pushed arbitrarily close to the prime
number theorem: for every $\varepsilon>0$ there exists $T(\varepsilon)$
such that knowledge of $\mu(n)$ for $n<T(\varepsilon)$ yields
$\limsup_{x\to\infty} \bigl|\pi(x)/(x/\log x)-1\bigr|<\varepsilon$.
Their proof, however, uses estimates equivalent to the prime number
theorem, so the method does not by itself give a new elementary proof
of the PNT but rather explains how far the Chebyshev approach can be
pushed in principle.

\section*{Acknowledgements}

The author would like to thank Jonah Saks for reading an early
draft of these notes, offering valuable feedback, and
independently verifying the computer programs during his undergraduate
research internship. The author is grateful to Alexandre Bailleul for
drawing his attention to the work of Diamond and Erd\H{o}s, and to
Andr\'e Pierro de Camargo for careful comments on the manuscript,
in particular for pointing out a missing nonnegativity hypothesis in
Theorem~\ref{t:cheb}(c) and for bringing the related work
\cite{CamargoMartin2022,CamargoMartin2024} to his attention.
This work was supported by the NSERC Discovery Grants Program.

\end{document}